\documentclass[12pt]{amsart}

\usepackage{amsthm}
\usepackage{amsfonts}
\usepackage{graphicx}

\newcommand{\C}{{\mathbb{C}}}

\newcommand{\R}{{\mathbb{R}}}

\newtheorem{theorem}{Theorem}[section]
\newtheorem{corollary}[theorem]{Corollary}
\newtheorem{lemma}[theorem]{Lemma}
\newtheorem{proposition}[theorem]{Proposition}
\newtheorem{conjecture}[theorem]{Conjecture}

\newtheorem{definition}[theorem]{Definition}

\numberwithin{figure}{section}
\numberwithin{theorem}{section}

\begin{document}

\author{Heather M. Russell and Julianna S. Tymoczko}
\address{Department of Mathematics, University of Iowa, Iowa City, 
IA 52242}
\email{tymoczko@math.uiowa.edu}
\email{hrussell@math.uiowa.edu}
\title{Springer representations on the Khovanov Springer varieties}

\maketitle

\begin{abstract}
Springer varieties are studied because their cohomology carries a natural action of the symmetric group $S_n$ and their top-dimensional cohomology is irreducible.  In his work on tangle invariants, Khovanov constructed a family of Springer varieties $X_n$ as subvarieties of the product of spheres $(S^2)^n$.  We show that if $X_n$ is embedded antipodally in $(S^2)^n$ then the natural $S_n$-action on $(S^2)^n$ induces an $S_n$-representation on the image of $H_*(X_n)$.  This representation is the Springer representation.  Our construction admits an elementary (and geometrically natural) combinatorial description, which we use to prove that the Springer representation on $H_*(X_n)$ is irreducible in each degree. We explicitly identify the Kazhdan-Lusztig basis for the irreducible representation of $S_n$ corresponding to the partition  $(n/2,n/2)$.
\end{abstract}

\section{Introduction}

Springer varieties are a collection of varieties that arise in two mathematical areas: geometric representation theory and low-dimensional topology.  Springer varieties were originally investigated because their cohomology carries a graded representation of the symmetric group $S_n$ and their top-dimensional cohomology is irreducible.  Many geometric constructions of the Springer representation exist; \cite{CG} has a survey.  One thing geometric constructions share is that explicit calculations are impossible outside of special cases.\footnote{Using an algebraic combinatorial construction of the cohomology ring, Garsia-Procesi identified the Springer representation in all degrees.}

In important recent work \cite{K2, K1} Khovanov studied a family of rings related to an invariant of tangle cobordisms.  He discovered the center of each ring was isomorphic to the cohomology of the Springer variety corresponding to the partition $(n/2,n/2)$, which we call $X_n$.  He conjectured that $X_n$ could be realized as a subspace of the $n$-fold product of spheres $(S^2)^{n}$, with components homeomorphic to $(S^2)^{n/2}$ and indexed by noncrossing matchings; this follows from recent work of Cautis-Kamnitzer \cite{Cautis-Kamnitzer} as shown in the Appendix of this paper.  Permutation representations arise in the work of Khovanov et alia, but for the group $S_{n/2}$ rather than $S_n$.  

This paper brings together the two themes of geometric representation theory and low-dimensional topology.  We obtain an explicit geometric and combinatorial construction of the Springer action for $X_n$.  Unlike other presentations of Springer representations, our construction is clear and elementary---both geometrically and combinatorially.  We use this explicit construction to identify the Springer representations in {\it every} degree of homology and to identify the Kazhdan-Lusztig basis explicitly and simply in top-degree homology.  

The variety $(S^2)^n$ and its homology have a natural $S_n$-action given by permuting the copies of $S^2$.  We use an antipodal map $\gamma: X_n \rightarrow (S^2)^{n}$ to embed $X_n$ in the product of spheres and prove that $S_n$ restricts to a permutation action on the image $\gamma_* (H_*(X_n))$.  Not every geometric map $\alpha: X_n \rightarrow (S^2)^n$ satisfies this property; indeed if $\alpha$ is the identity embedding then $\alpha_* (H_*(X_n))$ is in general not an $S_n$-subrepresentation of $H_*((S^2)^n)$.  

We prove four main theorems in this paper:
\begin{itemize}
\item that there is a well-defined $S_n$-action on $\gamma_*(H_*(X_n))$ (Theorem \ref{well-defined representation}); 
\item that the geometric representation is described combinatorially as an action of $S_n$ on noncrossing matchings that permutes the strands of the matchings (Theorem \ref{well-defined representation}); 
\item that the $S_n$-action on $\gamma_* (H_*(X_n))$ is the Springer representation (Theorem \ref{springer repn}); 
\item and that the Springer representation on $H_k(X_n)$ is the irreducible representation corresponding to the partition $(n-k,k)$ (Theorem \ref{springer repn}).  
\end{itemize}

A search of the literature shows that ``the Springer representation" can mean any of: Springer's original representation on the cohomology of Springer varieties \cite{Springer}, any other construction of the same representation, 
a construction of Springer's representation tensored with the sign representation \cite{Lusztig, Slodowy, Borho-Macpherson}, 
or a homology representation that is isomorphic to Springer's cohomology representation \cite{KL}.  (Hotta proves many of these equivalences in \cite{Hotta}.)  

Our construction gives an $S_n$-representation on homology with complex coefficients.  
We prove it is the Springer representation by comparing it explicitly with the de Concini-Procesi construction of Springer representations \cite{DeConciniProcesi}, using tools from Garsia-Procesi's combinatorial analysis of de Concini-Procesi's work \cite{GP}.  Our description of the homology of the Springer variety $X_n$ uses techniques from skein modules developed by Russell in \cite{R}, unlike de Concini-Procesi's presentation, which relies on Borel's quotient presentation of the cohomology of the flag variety.  

We exploit two different antipodal maps $X_n \rightarrow (S^2)^n$ in Section \ref{antipodal embeddings}.  The first is a geometric embedding of the entire variety, defined to be the identity on even coordinates and the antipodal map on odd coordinates.  The second is an embedding of each component of $X_n$; it is more natural in topological contexts and useful for computations but does not extend to all of $X_n$.  Corollary \ref{finally} proves the image $\gamma_*(H_*(X_n))$ is the span of the images of antipodal maps on each component, permitting us to use both types of antipodal maps in our calculations. 

The product basis of $H_*((S^2)^n)$ allows us to identify the $S_n$-representation on $\gamma_*(H_*(X_n))$ with a subspace of Young tabloids, a classical construction of $S_n$-representations.  The tabloid sums corresponding to basis elements in $\gamma_*(H_*(X_n))$ are not the classical basis for Specht modules; our methods allow us to conclude that we obtain the Kazhdan-Lusztig basis in top degree, and we conjecture that this is true in all degrees.  We show that for each $k$ the appropriate Specht module intersects $\gamma_*(H_k(X_n))$ nontrivially.  The dimensions of the two subspaces agree, so the image $\gamma_*(H_k(X_n))$ is irreducible and the $S_n$-representation we defined on $H_*(X_n)$ is the Springer representation.

We address this paper to several audiences: geometric representation theorists, low-dimensional topologists, and classical representation theorists.  Consequently we include material that each audience may find remedial.

The authors thank Charlie Frohman for stimulating conversations, Gary Kennedy and Tadeusz Januszkiewicz for helpful comments, Mikhail Khovanov for useful input, and Zhuojie Li, who made the observation that  inspired this work in an REU project.  The first author was partially supported by NSF VIGRE grant DMS-0602242; the second author was partially supported by NSF grant DMS-0801554.

\section{The Springer fiber $X_n$} \label{Springer topology}

\subsection{Classical results about the Springer action on $H_*(X_n)$}

Springer varieties are subvarieties of the flag variety, which is the collection of nested subspaces $V_1 \subseteq V_2 \subseteq \cdots \subseteq V_{n-1} \subseteq \C^n$ so that each $V_i$ is an $i$-dimensional subspace of $\C^n$.  A {\em nilpotent} linear operator $X: \C^n \rightarrow \C^n$ is a linear operator whose only eigenvalue is zero.  The Springer variety associated to a nilpotent linear operator  $X$ is the collection of flags fixed by $X$, namely 
\[\{\textup{Flags }V_{\bullet}: X V_i \subseteq V_i \textup{ for all } i\}.\]
If two linear operators are conjugate then the corresponding Springer varieties are homeomorphic.  The conjugacy class corresponding to $X$ is determined by the partition of $n$ given by the Jordan blocks of $X$.  We only study representatives of the homeomorphism class of a Springer variety, so for our purposes Springer varieties are indexed by partitions.

Let $\lambda = (\lambda_1, \ldots, \lambda_n)$ and $\mu= (\mu_1, \ldots, \mu_n)$ be partitions of $n$ written in decreasing order $\lambda_1 \geq \lambda_2 \geq \cdots \lambda_n \geq 0$ with some parts possibly zero.  We often omit terminal zeros in our notation.  We also interchangeably use the vocabulary for partitions and for Young diagrams (collections of left- and top-aligned boxes so that the $i^{th}$ row has $\lambda_i$ boxes for each $i$); hence $\lambda_i$ can be called either the $i^{th}$ part of the partition or the $i^{th}$ row of the Young diagram.  The Springer variety corresponding to $\lambda$ is denoted $X_{\lambda}$.  The irreducible representation corresponding to $\lambda$ is discussed in a later section; we often call the representation $\lambda$ as well. 

We use three facts about the Springer representation.  The first is a classical result of Springer theory.

\begin{proposition} {\bf (Springer)}
The $S_n$-module $H^{\textup{top}}(X_{\lambda})$ is the irreducible representation corresponding to $\lambda$.
\end{proposition}

Moreover the dimension of $X_{\lambda}$ is determined in \cite{Spalt} (among others) to be
\begin{equation}\label{dimension formula}
\sum_{i=1}^n (i-1)\lambda_i.
\end{equation}

The second was proven by De Concini-Procesi in \cite[page 213]{DeConciniProcesi}:

\begin{proposition} {\bf (De Concini-Procesi)}
The multiplicity of the irreducible representation $\mu$ in the {\em ungraded} representation $H^*(X_{\lambda})$ is the number of fillings of $\mu$ with the integers $\{1^{\lambda_1}, 2^{\lambda_2}, \ldots, n^{\lambda_n}\}$ so that rows increase weakly and columns increase strictly.
\end{proposition}

Thus we deduce the following lemma.

\begin{lemma} \label{two-row multiplicities}
If $\lambda = (n-j,j)$ is a two-row partition then the multiplicity of the irreducible representation $\mu$ in $H^*(X_{\lambda})$ is either zero or one.  It is one if $\mu$ is a two-row partition of $n$ with $\mu_1 \geq \lambda_1$; it is zero otherwise.
\end{lemma}

\begin{proof}
The Young diagram $\mu$ cannot be filled with $\{1^{\lambda_1}, 2^{\lambda_2}\}$ and have strictly-increasing columns unless $\mu$ has at most two rows and the second row has size at most $\mu_2 \leq \lambda_2$.  In this case there is exactly one filling with strictly-increasing columns and weakly-increasing rows, namely the filling with $1^{\lambda_1}, 2^{\lambda_2 - \mu_2}$ on the top row and $2^{\mu_2}$ on the second row.
\end{proof}

The third fact was proven by Garsia-Procesi in \cite[Equation 4.2]{GP}:

\begin{proposition} {\bf (Garsia-Procesi)}
Suppose that $\lambda$ and $\mu$ are partitions so that for each $i$ the sums $\mu_1 + \mu_2 + \cdots + \mu_i \geq \lambda_1 + \lambda_2 + \cdots + \lambda_i$.  Then there is a graded $S_n$-equivariant surjection $H^*(X_{\lambda}) \rightarrow \hspace{-.7em} \rightarrow H^*(X_{\mu})$.
\end{proposition}

From this together with the previous we conclude:

\begin{proposition} \label{GP identification of Springer rep}
Let $\lambda$ be any partition of $n$ with at most two rows.  Then $H^k(X_{\lambda})$ is the irreducible representation corresponding to the partition $(n-k,k)$ for each $k \leq \lambda_2$.
\end{proposition}

\begin{proof}
The proof is by induction on $\lambda_2$.  If $\lambda_2=0$ then the Springer variety $X_{\lambda}$ consists of a single flag, and $H^k(X_{\lambda})$ is the trivial representation when $k=0$ and is zero otherwise. 
This proves the base case.

Now assume the claim holds for the partition $\lambda' = (n-j+1, j-1)$ and let $\lambda = (n-j,j)$.  The variety $X_{\lambda}$ has dimension $\lambda_2$ by the formula in Equation \eqref{dimension formula}. The top-dimensional cohomology $H^{\lambda_2}(X_{\lambda})$ is the irreducible representation corresponding to $\lambda$ by the first property of Springer varieties.  Let $m_{\mu, \lambda}$ denote the multiplicity of the irreducible representation $\mu$ in $H^*(X_{\lambda})$ and similarly for $m_{\mu, \lambda'}$.  The cohomology surjection $H^k(X_{\lambda}) \rightarrow \hspace{-0.7em} \rightarrow H^k(X_{\lambda'})$ is an $S_n$-equivariant map for each $k$.  Hence $m_{\mu, \lambda} \geq m_{\mu, \lambda'}$ for each irreducible $\mu$.  Lemma \ref{two-row multiplicities} proved that $m_{\mu, \lambda} = m_{\mu, \lambda'}$ for each $\mu$ with $\mu_2 \leq \lambda'_2$, so the cohomology surjection must be an isomorphism when $k < j$.  Induction proves the claim.
\end{proof}

The number of standard fillings of $(n-k,k)$ is the rank of the irreducible representation corresponding to the partition $(n-k,k)$.  Hence we obtain a corollary.

\begin{corollary}\label{numfill}
If $\lambda$ is a two-row partition $(n-j,j)$ and $k \leq j$ then the rank of $H^k(X_{\lambda})$ is the number of standard fillings of $(n-k,k)$, namely fillings of $(n-k,k)$ with $\{1,2,\ldots,n\}$ without repetition so that both rows and columns increase.
\end{corollary}

Most of the classical theory of Springer representations is written for cohomology rather than homology.  The next lemma is included for the sake of completeness.

\begin{lemma}
If $k$ satisfies $0\leq k \leq n$ then over complex coefficients $H^k(X_n)\cong H_k(X_n)$. 
\end{lemma}

\begin{proof}
From Corollary \ref{numfill} we know $H^k(X_n) \cong \mathbb{C}^{b_k}$ where $b_k$ is the number of standard fillings of $(n-k, k)$. Say that $H_k(X_n) \cong \mathbb{C}^{a_k}$ for some nonnegative integer $a_k$. By definition  $H^k(X_n) = Hom_{\mathbb{C}}(H_k(X_n), \mathbb{C})$.  Since  $Hom_{\mathbb{C}}(\mathbb{C}^{a_k}, \mathbb{C}) \cong \left(Hom_{\mathbb{C}}(\mathbb{C}, \mathbb{C}) \right) ^{a_k} \cong \mathbb{C}^{a_k}$ we conclude $a_k = b_k$ for all $k$.
\end{proof}

\subsection{The Khovanov construction}\label{KhovCon}

Let $n\in \mathbb{N}$ be even. In \cite{K2} Khovanov proves for $\lambda = (n/2, n/2)$ the cohomology ring $H^*(X_{\lambda})$ is
isomorphic to the center of a ring 
related to his tangle invariant by proving both are isomorphic to the cohomology ring of a certain subspace $\widetilde{S}$ of $(S^2)^{n}$. Results from recent work of Cautis-Kamnitzer \cite{Cautis-Kamnitzer} imply that 
$X_{\lambda}$ and $\widetilde{S}$ are homeomorphic; the Appendix has a complete proof.  This section reviews the construction of $\widetilde{S}$, which we call the Khovanov construction of the $(n/2, n/2)$ Springer variety, and gives a diagrammatic set of generators for $H_*(\widetilde{S})$.  

\begin{definition}
A noncrossing matching on $n$ vertices is a collection of $n/2$ disjoint arcs which connect $n$ vertices (on a straight line) pairwise. 
\end{definition}
Let $B^{n/2}$ be the set of all noncrossing matchings on $n$ vertices. 
Given $a\in B^{n/2}$, let $S_a=\{ (x_1, \ldots, x_{n})\in (S^2)^{n} : x_i=x_j \text{ if } (i,j)\in a\}.$
Khovanov's construction of the $(n/2, n/2)$ Springer variety, denoted $\widetilde{S}$, is the union of $S_a$ over all matchings $$\widetilde{S} = \bigcup_{a\in B^{n/2}} S_a$$

Generators for the homology $H_*(\widetilde{S})$ can be described completely combinatorially.

\begin{definition}
A dotted noncrossing matching is a noncrossing matching in which some arcs have a single dot each.  We define the dotted arcs of the matching to be those arcs with dots on them.
\end{definition}

We will use $a,b$ to denote noncrossing matchings and $M, M'$ to denote dotted noncrossing matchings. There are many possible dotted noncrossing matchings associated to each noncrossing matching. Figure \ref{homgen} gives an example.  
\begin{figure}[h]
\includegraphics[width=1.2in]{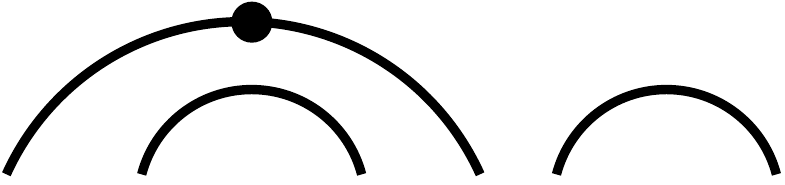}
\caption{A dotted noncrossing matching (for the generator $[p\times c \times c]$)}\label{homgen}
\end{figure}
Dotted noncrossing matchings are a simple, diagrammatic way to represent  the homology generators corresponding to cells in the cartesian product CW-decomposition for $S_a$.  This was first described in \cite[Section 5.2]{R}.

\begin{proposition} {\bf (Russell)}
The different dotted noncrossing matchings for the noncrossing matching $a$ are naturally bijective with the homology generators for $H_*(S_a)$.
\end{proposition}

\begin{proof}The standard CW-decomposition of the two-dimensional sphere $S^2$ is given by a fixed point $p  \in S^2$ and the two-cell $c = S^2 - p$.  A cartesian product of spheres is endowed with the cartesian product cell decomposition and homology basis.  Since $S_a$ is diffeomorphic to $(S^2)^{n/2}$ we may fix such a cell decomposition.  Each arc in a matching $a$ represents one sphere in the cartesian product $S_a$. Each generator of $H_*(S_a)$ consists of the choice of a point or a two-cell for each of these spheres. Given a dotted noncrossing matching $M$ associated to $a$, the corresponding homology generator chooses a point for the homology generator of a sphere if and only if the corresponding arc in $M$ has a dot. 
\end{proof}

Because of this proposition we will often refer to the dotted noncrossing matchings and the corresponding homology generators interchangeably.  Figure \ref{homgen} shows an example of a generator for $H_*(S_a)$ where $a$ is the matching on 6 vertices. 

Varying over all possible markings of a single matching yields a generating set for $H_*(S_a)$. By considering all possible markings of all noncrossing matchings on $n$ vertices, we get a set of generators for $\bigoplus_a H_*(S_a)$.  

In \cite[Section 3.2]{R} Russell defines Type I and Type II relations on dotted noncrossing matchings.
\begin{definition} \label{relations}
Let $a,b\in B^{n/2}$  with 
all arcs in $a$ and $b$ identical except $(i,j)$, $(k,l) \in a$ and $(i,l)$, $(j,k)\in b$ for some $i<j<k<l$.  Fix a choice of dots on the arcs of $a$ in positions other than $i,j,k,l$.
\begin{enumerate}
\item {\bf Type I relations:}  Let $M_1$ be the dotted noncrossing matching with dotted arc $(i,j)$  and undotted arc $(k,l)$.  Let $M_2$ be the dotted noncrossing matching with dotted arc $(k,l)$ and undotted arc $(i,j)$.  Define $M'_1$ and $M'_2$ analogously for $b$, so that they agree with the $M_i$ off of $i,j,k,l$.  Type I relations have the form
\[M_1+M_2-M'_1-M'_2=0.\]
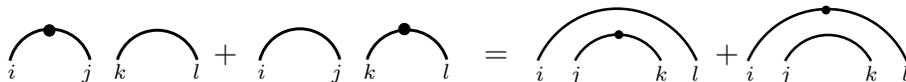
\begin{figure}[h]
\begin{equation*}
 \begin{picture}(85,10)(0,0)
 \put(10,0){\includegraphics[width=1in]{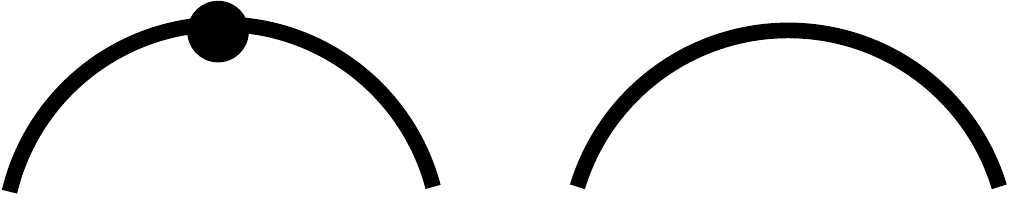}}
 \put(10,-6){\tiny{$i$}}
 \put(38,-6){\tiny{$j$}}
 \put(50,-6){\tiny{$k$}}
 \put(80,-6){\tiny{$l$}}
 \end{picture} + 
 \begin{picture} (80,10)(0,0)
 \put(5,0){\includegraphics[width=1in]{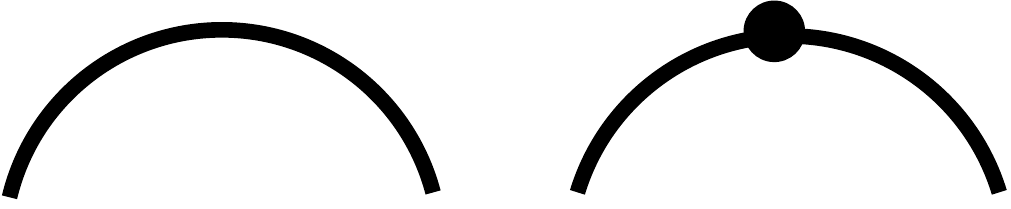}}
 \put(5,-6){\tiny{$i$}}
 \put(33,-6){\tiny{$j$}}
 \put(45,-6){\tiny{$k$}}
 \put(75,-6){\tiny{$l$}}
 \end{picture} \hspace{.1in} = \hspace{.1in} 
 \begin{picture}(65,10)(0,0)
 \put(0,0){\includegraphics[width=.85in]{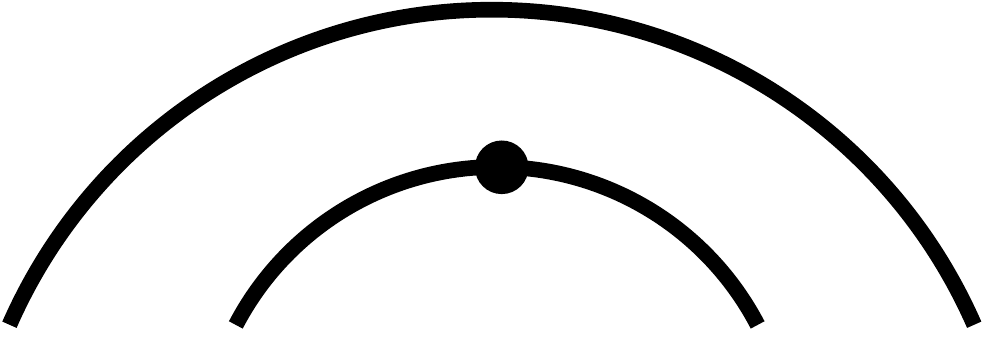}}
 \put(0,-7){\tiny{$i$}}
 \put(13,-7){\tiny{$j$}}
 \put(45,-7){\tiny{$k$}}
 \put(59,-7){\tiny{$l$}}
 \end{picture} + 
 \begin{picture}(80,10)
 \put(0,0){\includegraphics[width=.85in]{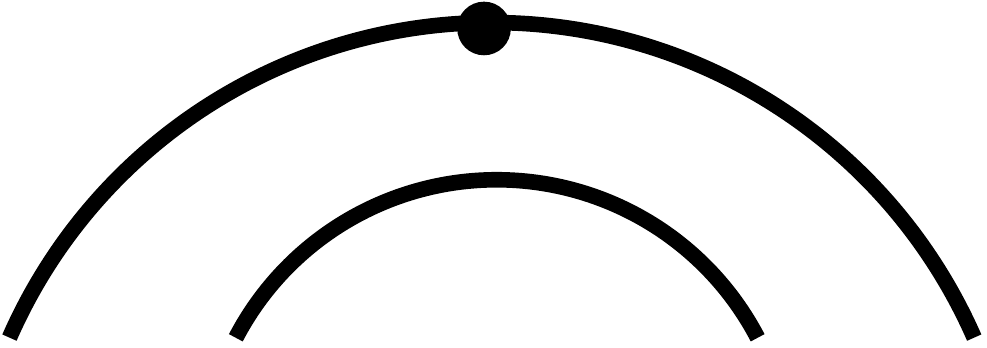}}
 \put(0,-7){\tiny{$i$}}
 \put(13,-7){\tiny{$j$}}
 \put(45,-7){\tiny{$k$}}
 \put(59,-7){\tiny{$l$}}
 \end{picture}
\end{equation*}
\caption{Type I relation}\label{tp1}
\end{figure}
\item {\bf Type II relations:} Let $M_3$ be the dotted noncrossing matching with dotted arcs $(i,j)$ and $(k,l)$. Define $M'_3$ analogously for $b$.  Type II relations have the form
\[M_3-M'_3=0.\]
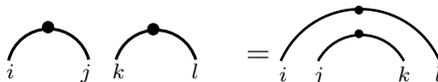
\begin{figure}[h]
\begin{equation*}
\begin{picture}(80,10)(0,0)
\put(0,0){\includegraphics[width=1in]{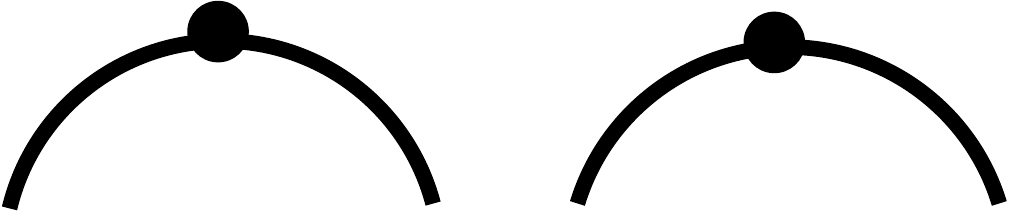}} 
\put(0,-6){\tiny{$i$}}
 \put(28,-6){\tiny{$j$}}
 \put(40,-6){\tiny{$k$}}
 \put(70,-6){\tiny{$l$}}
\end{picture} \hspace{.1in} = 
\begin{picture}(80,10)(0,0)
\put(0,0){\includegraphics[width=.85in]{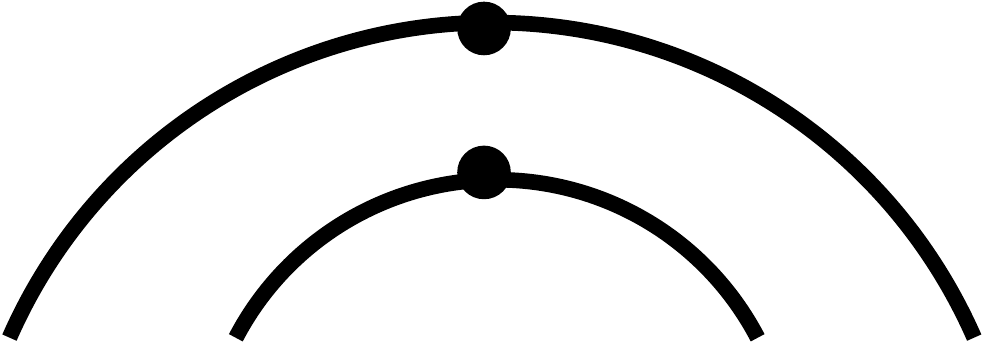}}
\put(0,-7){\tiny{$i$}}
 \put(13,-7){\tiny{$j$}}
 \put(45,-7){\tiny{$k$}}
 \put(59,-7){\tiny{$l$}}
\end{picture}
\end{equation*}
\caption{Type II relation}\label{tp2}
\end{figure}
\end{enumerate}
\end{definition}

Appealing to Type I and Type II relations \cite[Theorem 3.1]{R} proved the following.
\begin{proposition}{\bf (Russell)}
Let $R$ be the subspace of $\bigoplus_a H_*(S_a)$ generated by all Type I and Type II relations. Then 
$$H_*(\widetilde{S}) \cong \left( \bigoplus_a H_*(S_a) \right) /R.$$
\end{proposition}


Each undotted arc corresponds to a two-cell and each dotted arc corresponds to a point in the cartesian product CW-decomposition of $S_a$, which describes the underlying topology of the relations.  Figures \ref{tp1} and \ref{tp2} show both types. Note that there is a fixed arrangement of dotted and undotted arcs compatible with the vertex labelings that is not shown.  

\subsection{Bijection between the homology basis for $X_n$ and standard tableaux} \label{bijection}

The generators for homology described in Section \ref{KhovCon} are not linearly independent.  In this section we choose a homology basis for $X_n$ indexed by {\em standard} noncrossing matchings.  We then describe a bijection between the standard noncrossing matchings and the standard tableaux of shape $(n-k,k)$ for $k=0,1,\ldots, n/2$.  

\begin{definition}
A standard noncrossing matching is a (dotted) noncrossing matching in which no dotted arc is nested below another arc.
\end{definition}

As a consequence, the only dotted arcs in a standard noncrossing matching are either entirely unnested or the topmost arc in a nesting.  Figures \ref{std} and \ref{nonstd} show several examples.

\begin{figure}[h]
 \raisebox{-12pt}{ \includegraphics[width=.8in]{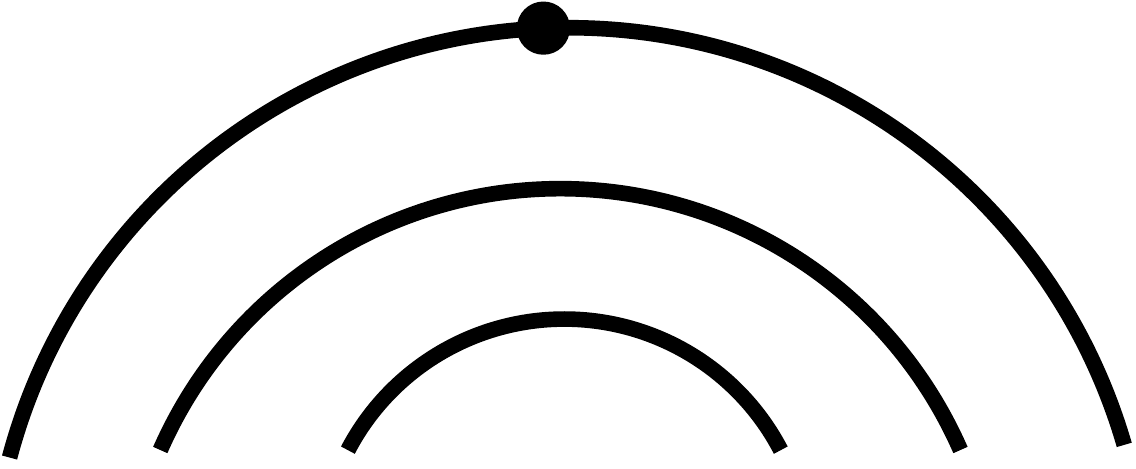}} \hspace{.5in} \raisebox{-8pt}{\includegraphics[width=1in]{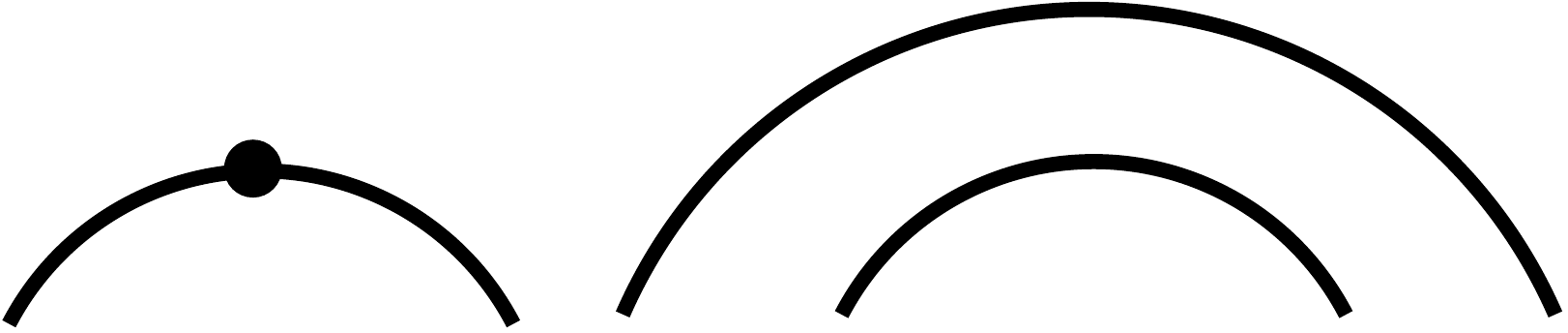}}
\caption{Standard dotted noncrossing matchings}\label{std}
\end{figure}

\begin{figure}[h]
\raisebox{-10pt}{\includegraphics[width=1in]{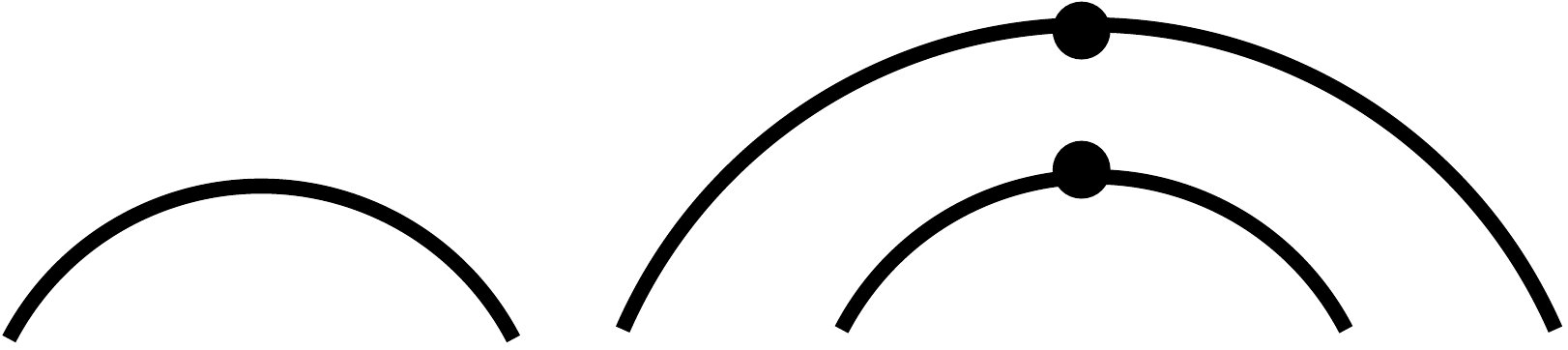}} \hspace{.5in} \raisebox{-12pt}{\includegraphics[width=.8in]{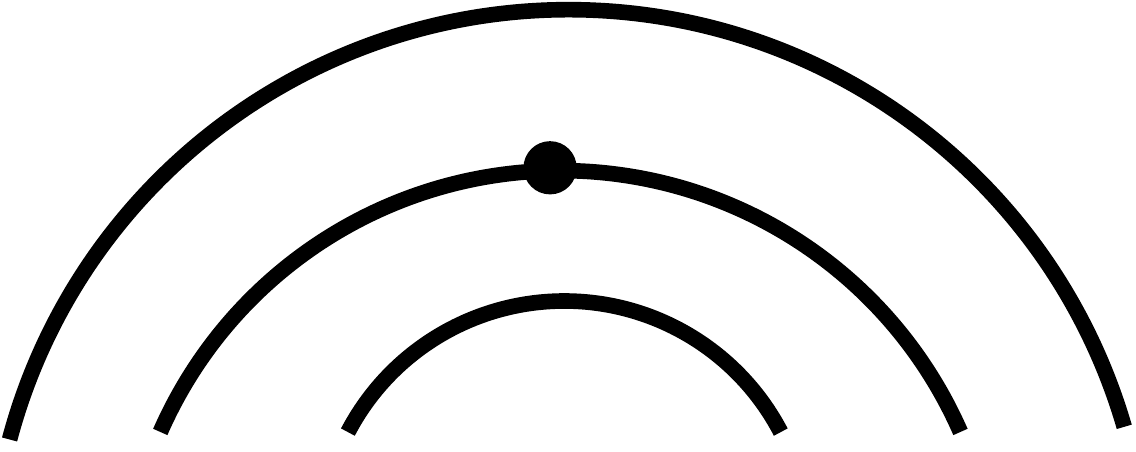}}
\caption{Nonstandard dotted noncrossing matchings}\label{nonstd}
\end{figure}

We now define a function $\varphi$ from dotted noncrossing matchings to standard tableaux.  Given a dotted noncrossing matching $M$, construct a tableau $T$ as follows:
\begin{itemize}
\item For each undotted arc $(i,j)$ with $i < j$, place the number $j$ in the bottom row of $T$.
\item Place all other numbers in the top row of $T$.
\item Order each row so that the numbers increase left-to-right.
\end{itemize}
The tableau $T$ obtained by this process is $\varphi(M)$.  It is a two-row standard tableau by construction.  Note that if the dotted matching $M$ had $k$ undotted arcs and $n$ vertices, then the resulting tableau has shape $(n-k,k)$, as shown in Figure \ref{tableau to dncm}.

\begin{figure}[h]
\raisebox{-4pt}{\includegraphics[width=1in]{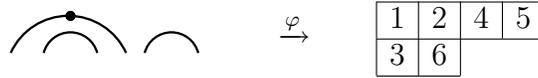}} \hspace{.3in} $\xrightarrow{\varphi} \hspace{.3in}
\begin{tabular}{|c|c|c|c|}
\cline{1-4} 1 & 2 & 4 & 5 \\
\cline{1-4} 3 & 6 & \multicolumn{2}{c}{} \\
\cline{1-2} \multicolumn{4}{c}{} \vspace{-1.3em} \end{tabular}$
\caption{A homology generator and its associated tableau}\label{tableau to dncm}

\end{figure}

Next we define a map $\theta$ from standard two-row tableaux to standard dotted noncrossing matchings.  Given a standard tableau $T$ of shape $(n-k,k)$, construct a standard dotted noncrossing matching $M$ as follows:
\begin{enumerate}
\item Repeat the next two steps until all numbers on the second row have been matched:
\begin{enumerate}
\item Let $j$ be the leftmost unmatched number in the second row.
\item Add an undotted arc from $j$ to the rightmost unmatched $i$ to the left of $j$ in $M$.  (In other words $i = \max \{ k: k < j, k \textup{ unmatched}\}$.)
\end{enumerate}
\item If $i$ is the leftmost unmatched number in $M$, add a dotted arc from $i$ to its nearest unmatched neighbor on the right.  (In other words, add a dotted arc $(i,j)$ where $j = \min \{ k: k > i, k \textup{ unmatched}\}$.)
\end{enumerate}
The matching $M$ is by definition $\theta(T)$.  Figure \ref{dncm to tableau} gives an example.

\begin{figure}[h]
$\begin{tabular}{|c|c|c|c|c|}
\cline{1-5} 1 & 2 &  3 & 4 & 6 \\
\cline{1-5}  5 & \multicolumn{4}{c}{} \\
\cline{1-1} \multicolumn{5}{c}{} \vspace{-1.3em} \end{tabular} \hspace{.3in}
\xrightarrow{\theta} \hspace{.3in} \raisebox{-4pt}{\includegraphics[width=1in]{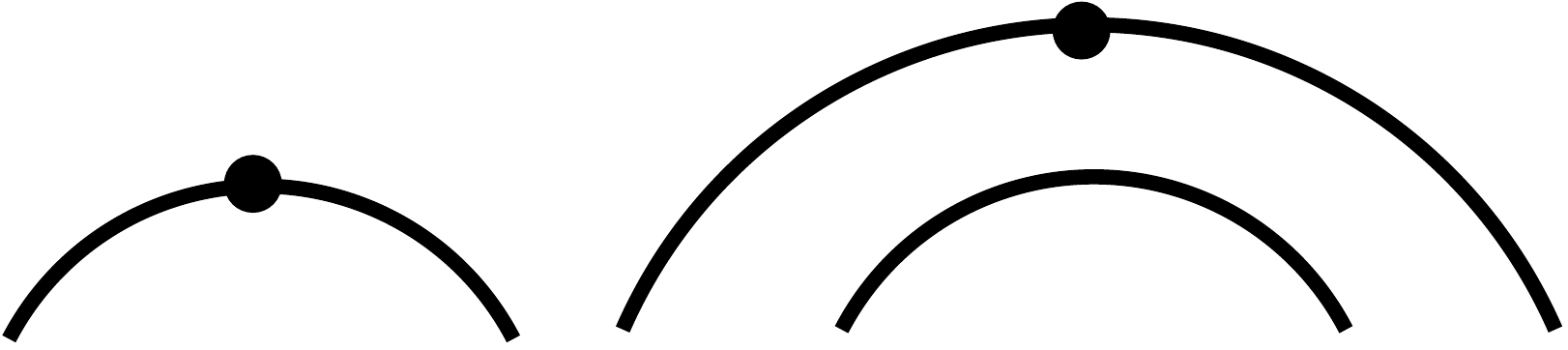}}$

\caption{A tableau of shape $(5, 1)$ and its associated matching}\label{dncm to tableau} 

\end{figure}

\begin{lemma}
For each standard two-row tableau $T$, the matching $M=\theta(T)$ is a standard dotted noncrossing matching.  If $T$ has shape $(n-k,k)$ then $M$ has exactly $k$ undotted arcs.
\end{lemma}

\begin{proof}
Denote the column of $j$ by $c(j)$.  Each of the numbers in the first $c(j)$ positions of the first row of $T$ are above or to the left of $j$.  The definition of standard tableau implies that all of these numbers are less than $j$.  Hence Step (b) can always be performed.

We confirm the matching produced by Step (1) is noncrossing.  Suppose arc $(i,j)$ was added after arc $(i',j')$.  Step (a) ensures that $j'$ is to the left of $j$ in the standard tableau so $j' < j$.  If the arcs cross then either $i' < i < j'$ or $j' < i'$.  Both contradict Step (b).

There are no nested unmatched vertices after Step (1), namely unmatched vertices $k$ with an arc $(i,j)$ so that $i < k < j$.  (Such $k$ contradict the choice of $i$ in Step (b).)  Hence Step (2) creates neither crossing arcs nor dotted arcs that are nested by undotted arcs.  (The latter would be a dotted $(i,j)$ such that there is an undotted $(i',j')$ with $i'<i<j<j'$.)  

This shows $\theta(T)$ is a standard dotted noncrossing matching.  Exactly one undotted arc was constructed for each box in the second row of $T$, which proves the claim.
\end{proof}

\begin{theorem}\label{bijection between tableaux and matchings}
Let $n,k$ be a pair with $n$ even and $k \leq n/2$.  The map $\varphi$ is a bijection between standard noncrossing matchings on $n$ vertices with exactly $k$ undotted arcs and two-row standard tableaux of shape $(n-k,k)$.  Its inverse is $\theta$.
\end{theorem}

\begin{proof}
The composition $\varphi \circ \theta$ is the identity since the right endpoints of the undotted arcs in $\theta(T)$ are exactly the second row of $T$.

If $M$ has no undotted arcs then $\varphi(M)$ is the (unique) tableau of shape $(n)$.  The matching $\theta(\varphi(M))$ is the entirely unnested, entirely dotted matching, which has a dotted arc $(2i-1,2i)$ for each $i \leq n/2$.  This is the only fully dotted standard noncrossing matching, since standard in this case implies that the matching has no nesting.  So $\theta(\varphi(M)) = M$ when $k=0$ (independent of $n$).

Assume that $\theta \circ \varphi$ is the identity on all standard noncrossing matchings with at most $k-1$ undotted arcs (for all $n$).  Let $M$ be a matching with $k$ undotted arcs, let $T = \varphi(M)$, and let $j$ be the leftmost vertex with an undotted arc $(i,j)$ in $M$.  Then $i = j-1$, since otherwise the arc $(i,j)$ nests another arc $(i',j')$ with $i < i' < j' < j$.  This either violates the minimality assumption on $j$ or the hypothesis that $M$ is standard. 

Let $M^j$ denote the matching obtained by removing the vertices $j-1$ and $j$, and renaming the remaining vertices $1, \ldots, n-2$.  Let $T^j$ be the standard tableau obtained by removing the entries $j-1$ and $j$ and renaming the rest of the entries $1, \ldots, n-2$.  Note that $\varphi(M^j) = T^j$ and that $M^j$ has $k-1$ undotted arcs.  The inductive hypothesis says that $\theta(\varphi(M^j)) = M^j$.  

We compare $\theta(T)$ to $\theta(T^j)$.  The first iteration of Step (1) on $T$ creates an undotted arc $(j-1,j)$ since $j$ is the first entry in the second row of $T$ and hence $j-1$ is on the first row.  Successive iterations create an undotted arc $(i,k)$ in $T$ if and only if either $i>j$ and $(i-2,k-2)$ is an arc in $T^j$ or $i<j-1$ and $(i,k-2)$ is an arc in $T^j$.  Hence $\theta(T)$ is the standard matching obtained by relabeling each vertex $k$ such that $k \geq j-1$ with the number $k+2$, and then inserting the undotted arc $(j-1,j)$ into $T^j$.  In other words $\theta(T)=T$.  This implies by induction that $\theta \circ \varphi$ is the identity on all standard noncrossing matchings.
\end{proof}

\begin{corollary}
The standard noncrossing matchings on $n$ vertices with $k$ undotted arcs form a basis for $H_{2k}(X_n)$.
\end{corollary}

\begin{proof}
The noncrossing matchings on $n$ vertices with $k$ undotted arcs generate $H_{2k}(X_n)$, so we must show that each nonstandard noncrossing matching is equivalent to a linear combination of standard noncrossing matchings with the same number of undotted arcs.  Type I relations (see Definition \ref{relations}) permit us to replace each instance of a dotted arc nested below an undotted arc by a linear combination of matchings with the dotted arcs at the top of nestings.  This means we may assume the matching is nonstandard only because dotted arcs are nested.  Type II relations permit us to replace each instance of two nested dotted arcs $(i,j)$ and $(i',j')$ satisfying $i'<i<j<j'$ with the unnested dotted arcs $(i,i')$ and $(j,j')$.  Together, these relations allow us to reduce to standard form.
\end{proof}

\section{The topological $S_n$-action on $H_*(X_n)$}

The space $(S^2)^n$ has a natural action of the permutation group $S_n$ given by permuting the components.  We will use an antipodal map to embed the Springer fiber $X_n$ into $(S^2)^n$. This gives a natural action of $S_n$ on the image of $H_*(X_n)$ under the induced map on homology.  In this section we describe this map and the $S_n$ action on $H_*(X_n)$, giving an explicit formula for the action of each simple transposition $s_i$ on the homology basis from Section \ref{Springer topology}.

\subsection{The antipodal embedding}\label{antipodal embeddings}

We will define two embeddings. The first is described as a map on the space $X_n$ but ignores the structure coming from the individual components $S_a$. The second embedding, defined for each $S_a$, is more convenient for computational purposes but does not extend to $X_n$. By showing that both maps give the same representation, we can use the first as a geometric embedding of the Springer variety and the second to study the $S_n$-action on that embedding.

Many key calculations involve $H_*((S^2)^n)$. As in the prior sections, we use a homology basis associated to the cartesian product cell decomposition. Line diagrams, like those shown in Figure \ref{lndg}, allow us to extend the diagrammatic notation from Section \ref{KhovCon} to the cartesian product generators of $H_*((S^2)^n)$.
 \begin{definition}
A line diagram is a collection of $n$ parallel, vertical line segments some of which are marked with a single dot. 
 \end{definition}
Each cartesian product generator for $H_*((S^2)^n)$ is represented by a dotted strand for each sphere where a point has been chosen and an undotted strand for each sphere where a two-cell has been chosen. Figure \ref{lndg} has an example.
 
\begin{figure}[h] 
\begin{picture}(70, 30) (0,0) 
\put(-30,5){\line(0,1){24}}
\put(-30, 17){\circle*{5}}
\put(-22,5){\line(0,1){24}}
\put(-14,5){\line(0,1){24}}
\put(-14, 17){\circle*{5}}
\put(-6,5){\line(0,1){24}}
\put(-50, -5){$[p\times c \times p \times c]$}
\put(60,5){\line(0,1){24}}
\put(68,5){\line(0,1){24}}
\put(76,5){\line(0,1){24}}
\put(84,5){\line(0,1){24}}
\put(84, 17){\circle*{5}}
\put(40, -5){$[c\times c \times c \times p]$}
\end{picture}
\caption{Generators for $H_*( (S^2)^{4})$}\label{lndg}
\end{figure}
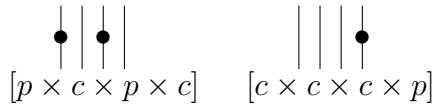

We begin building the antipodal embeddings by defining two maps $f$ and $g$.  There is only one noncrossing matching on two vertices, namely the arc $(1,2)$. According to Khovanov's construction the space $X_2 = \{ (x,x): x\in S^2\}$. Define maps $f,g: X_2 \rightarrow S^2\times S^2$ by
$$ f(x,x) = (x,-x) \hspace{0.5in} \textup{ and } \hspace{0.5in} g(x,x) = (-x, x)$$
where $-x$ denotes the image of $x$ under the antipodal map $S^2 \rightarrow S^2$.  (If $S^2$ is realized as the solutions in $\R^3$ to $x_1^2+x_2^2+x_3^2 = 1$ then the antipodal map does in fact send $x$ to $-x$ for each $x \in S^2$.)

Since $X_2 \cong S^2$ the homology $H_*(X_2)$ is generated by a point $p$ and a two-cell $c$. The diagrams associated to these generators are shown in Figure \ref{x2hom}. 
\begin{figure}[h]
\begin{picture}(40, 20) (0,0)
\put(-40, 0){\includegraphics[width=.4in]{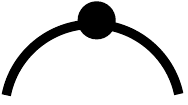}}
\put(-32, -10){$[p]$}
\put(49, -10){$[c]$}
\put(40,0){\includegraphics[width=.4in]{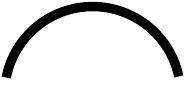}}
\end{picture}
\caption{Generators for $H_*(X_2)$}\label{x2hom}
\end{figure}
\begin{lemma}\label{fgmap}
The image in homology of the maps $f$ and $g$ is:
$$ f_*([p]) = [p\times p] = g_*([p])  \hspace{0.5in} \textup{ and } \hspace{0.5in} f_*([c]) = [c\times p] - [p\times c] = -g_*([c]).$$
\end{lemma}

\begin{proof}
We first consider these maps applied to the zero cell $p$.  We get $f(p) = (p, -p)$ and $g(p) = (-p, p)$. All points are homologous in a connected space so we obtain
\[ [(p,p)] = [(p, -p)] = [(-p,p)] \in H_*(S^2 \times S^2).  \] 
Then  $f_*([p]) = [p\times p] = g_*([p])$ as desired. In terms of diagrams this is  
$$f_*(\includegraphics[width=.3in]{x2dot.pdf}) = \begin{picture}(16, 10) (0,0)
\put(5,-5){\line(0,1){20}}
\put(5, 5){\circle*{5}}
\put(11,-5){\line(0,1){20}}
\put(11, 5){\circle*{5}}
\end{picture}
= g_*(\includegraphics[width=.3in]{x2dot.pdf}).$$

We now find $f_*([c])$ and $g_*([c])$.  Let $\pi_1, \pi_2: S^2 \times S^2 \rightarrow S^2$ denote projection to the respective coordinate and let $\iota_1, \iota_2: S^2 \hookrightarrow S^2 \times S^2$ denote inclusion to the respective coordinate, for instance $\iota_1(x) = (x,p)$ for all $x \in S^2$ and a fixed $p \in S^2$.  

Since $\iota_1(c) = c \times p$ we have $(\iota_1)_*([c]) = [c \times p]$ and similarly for $\iota_2$.  In addition $\pi_1 \circ \iota_1 = id$ and $(\pi_1 \circ \iota_2)(x)= p$ for all $x \in S^2$.  Hence $(\pi_1 \circ \iota_1)_*([c]) = [c]$ and $(\pi_1 \circ \iota_2)_*([c]) = 0$.  We conclude that $(\pi_1)_*([c \times p]) = [c]$ and $(\pi_1)_*([p \times c]) = 0$.  A symmetric argument proves $(\pi_2)_*([c \times p]) = 0$ and $(\pi_2)_*([p \times c]) = [c]$. 

Define $f_0: S^2 \rightarrow X_2$ by $f_0(x) = (x,-x)$.  Note that $f_0(c) = f(c)$ and so $(f_0)_*([c]) = f_*([c])$.  We compute the constants $c_1, c_2$ in 
\[(f_0)_*([c]) = c_1 [c \times p] + c_2 [p \times c] \] 
using the maps $\pi_1 \circ f_0$ and $\pi_2 \circ f_0$.  The composition $\pi_1 \circ f_0 = id$ so $(\pi_1 \circ f_0)_*([c]) = [c]$.  Applying each map sequentially gives 
\[(\pi_1)_*((f_0)_*([c])) = (\pi_1)_*(c_1 [c \times p] + c_2 [p \times c]) = c_1 [c]\]
and so $c_1 = 1$.  The composition $\pi_2 \circ f_0$ is the antipodal map $\alpha$ and $\alpha_*([c]) = -[c]$.  
Hence 
\[ (\pi_2)_*(c_1 [c \times p] + c_2 [p \times c]) = c_2 [c] = \alpha_*([c]) \]
so $c_2 = -1$.  Analogous arguments show $(g_0)_*([c]) = -(f_0)_*([c])$.  In terms of diagrams this is 
$$f_*(\includegraphics[width=.3in]{x2.pdf}) = \begin{picture}(14, 10) (0,0)
\put(5,-5){\line(0,1){20}}
\put(11, 5){\circle*{5}}
\put(11,-5){\line(0,1){20}}
\end{picture}
- 
\begin{picture}(14, 10) (0,0)
\put(3,-5){\line(0,1){20}}
\put(9,-5){\line(0,1){20}}
\put(3, 5){\circle*{5}}
\end{picture}
= -g_*(\includegraphics[width=.3in]{x2.pdf}). $$
\end{proof}

For fixed (even) $n$ and each $0\leq m\leq \frac{n}{2}$, define $G^m: (X_2)^{\frac{n}{2}} \rightarrow (S^2)^n$ to be $$ G^m = g\times  \cdots \times g \times f \times  \cdots \times f$$ where $g$ is applied to the first $m$ copies of $X_2$ and $f$ is applied to the last $n/2-m$ copies. Let $a_0$ be the completely unnested matching, namely  the matching with arcs $(1,2), \ldots, (n-1, n)$. Note that $S_{a_0} = (X_2)^{n/2}$ so the maps $G^m$ can be thought of as having domain $S_{a_0}$. 

\begin{lemma}\label{Fmaps}
For all $0\leq m \leq \frac{n}{2}$ and cartesian product generators $q \in H_*((X_2)^{\frac{n}{2}})$ the maps $G_*^m$ satisfy $G_*^m(q) = \pm G_*^{n/2}(q)$.
\end{lemma}
\begin{proof}
Consider the cartesian product decomposition on $(X_2)^{n/2} = S_{a_0}$ and let $q\in H_*(S_{a_0})$ be a cartesian product generator. Then $q$ has the form $q = [q_1 \times  \cdots \times q_{n/2}]$ where each $q_i = p$ or $c$. Since $H_*((S^2)^n)) \cong \bigotimes_{i=1}^{n/2} H_*(S^2\times S^2) $ we have 
\begin{eqnarray*}
 G_*^m(q) & = &  (g\times \cdots \times g \times f \times f)_*(q)\\
& = & g_*(q_1)\otimes \cdots \otimes g_*(q_m) \otimes f_*(q_{m+1}) \otimes \cdots \otimes f_*(q_{n/2})
\end{eqnarray*}

Lemma \ref{fgmap} showed $f_*(q_i) = \pm g_*(q_i)$ for all $i$, and thus
$$G_*^m(q) = g_*(q_1)\otimes \cdots \otimes g_*(q_m) \otimes \pm g_*(q_{m+1}) \otimes \cdots \otimes \pm g_*(q_{n/2}) = \pm G_*^{n/2}(q).$$ 
\end{proof}

Define the embedding $\gamma: X_n \rightarrow (S^2)^n$ as $$\gamma(x_1, \ldots, x_n) = (-x_1, x_2, \ldots, -x_{n-1}, x_n) = ((-1)^{i}x_i)$$ where $-x$ again denotes the image of $x$ under the antipodal map $S^2 \rightarrow S^2$.  For each matching $a\in B^{n/2}$, define an embedding of the component $\gamma_a: S_a \rightarrow (S^2)^{n/2}$ to be the identity on coordinates corresponding to right endpoints of arcs and the antipodal map on coordinates corresponding to left endpoints. Explicitly let $\rho_{a}$ be the binary function $$\rho_{a}(i) = 
\begin{cases}
1 \text{ if there is $j$ with } (i,j)\in a\\
0 \text{ if there is $j$ with } (j,i) \in a
\end{cases}$$ 
Then  $\gamma_a: S_a \rightarrow (S^2)^n$ is defined to be $\gamma_a(x_1, \ldots, x_n) = ((-1)^{\rho_a(i)}x_i).$

For example let $n=6$ and let $a$ be the matching in Figure \ref{matex}.
 \begin{figure}[h]
\includegraphics[width=.8in]{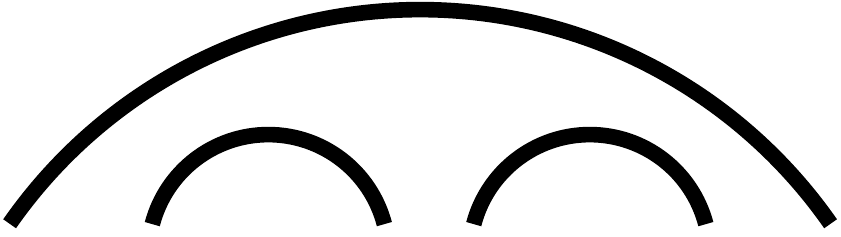}
\caption{The matching $a$}\label{matex}
\end{figure}
Then each $\vec{v} \in S_a$ has the form $\vec{v} = (x,y,y,z,z,x)$ where $x,y,z\in S^2$. The antipodal maps $\gamma$ and $\gamma_a$ applied to these elements give  
$$\gamma (\vec{v}) = (-x,y,-y,z,-z,x) \hspace{.3in} \text{  and} \hspace{.3in} \gamma_a(\vec{v}) = (-x, -y, y, -z, z, x) .$$
The map $\gamma_a$ is related to the parity of coordinates, as the next lemma shows.

\begin{lemma}\label{parity}
Each arc in a noncrossing matching has one even and one odd endpoint.
\end{lemma}
\begin{proof}
Take $a\in B^{n/2}$ and an arc $(i,j) \in a$. In order for $a$ to be noncrossing, all vertices between $i$ and $j$ must be paired among themselves. This means that there are an even number of vertices between $i$ and $j$, so $i$ and $j$ have opposite parities. 
\end{proof}

Again let $a_0$ be the matching with arcs $(1,2), (3,4), \ldots, (n-1, n)$. Let $a\in B^{n/2}$ be any matching and let $m$ be the number of arcs in $a$ with odd left endpoints. Order the arcs in $a$ by listing all arcs with odd left endpoints first, and otherwise listing from least left endpoint to greatest left endpoint.  This gives a list $e_1, \ldots,  e_{n/2}$ where the arcs $e_1, \ldots , e_{m}$ have odd left endpoints and the arcs $e_{m+1}, \ldots , e_{n/2}$ have even left endpoints. Define $\sigma_a$ to be the permutation that maps the arc $(2k-1, 2k)$ to the arc $e_k$ so that $\sigma_a(2k-1)$ is the left endpoint of $e_k$ and $\sigma_a(2k)$ is the right endpoint.  Hence $\sigma_a$ sends the matching $a_0$ to $a$. Let $\sigma_a$ also denote the map that permutes coordinates of  $(S^2)^n$ so that
$$\sigma_a(x_1, \ldots , x_n) = (x_{\sigma_a^{-1}(1)}, \ldots , x_{\sigma_a^{-1}(n)}).$$ 

For example consider the matching in Figure \ref{matex}. There is only one arc with an odd left endpoint so $m=1$. The ordered list of arcs in $a$ is $(1,6), (2,3), (4,5)$ and the permutation is $\sigma_a = (2 6 5 4 3)$ in cycle notation. For $x,y,z\in S^2$ the map $\sigma_a: S_{a_0} \rightarrow S_a$ will be
$$ (x,x,y,y,z,z) \mapsto (x,y,y,z,z,x).$$
\begin{lemma}\label{comp1}
Choose a matching $a\in B^{n/2}$ and let $m$ be the number of arcs in $a$ with odd left endpoints. Then the maps $\gamma|_{S_a}$ and $\sigma_a \circ G^m \circ \sigma_a^{-1}$ from $S_a$ to $(S^2)^n$ are equal.
\end{lemma}

\begin{proof}
Choose $i$ with $1\leq i \leq n$. The map $\sigma_a^{-1}$ sends arcs in $a$ with odd left endpoints to the arcs $(1,2), (3,4), \ldots, (2m-1,2m)$ and arcs with even left endpoints to $(2m+1,2m+2), \ldots, (n-1,n)$.  The map $G^m$ acts as $g$ on the first $m$ arcs and as $f$ on the last $n/2-m$ arcs; this only affects the sign of each coordinate.  Then $\sigma_a$ returns each coordinate to its original position.

By Lemma \ref{parity}  each arc has one odd and one even endpoint.  Hence if $i$ is odd and a left endpoint in $a$ then $\sigma_a^{-1}$ sends $i$ to one of $1,3,5,\ldots, 2m-1$; if $i$ is odd and a right endpoint in $a$ then $\sigma_a^{-1}$ sends $i$ to one of $2m+2, 2m+4, \ldots, n$.  In either case
$$x_i \overset{\sigma_a^{-1}}{\rightarrow} x_{\sigma_a^{-1}(i)} \overset{G^m}{\rightarrow} - x_{\sigma_a^{-1}(i)}  \overset{\sigma_a}{\rightarrow} - x_i.$$
Similarly if $i$ is even and either a left or right endpoint in $a$ then
$$x_i \overset{\sigma_a^{-1}}{\rightarrow} x_{\sigma_a^{-1}(i)} \overset{G^m}{\rightarrow} x_{\sigma_a^{-1}(i)}  \overset{\sigma_a}{\rightarrow} x_i.$$
The composition $\sigma_a \circ G^m \circ \sigma_a^{-1}$ is the identity on even coordinates and the antipodal map on odd coordinates so $\gamma|_{S_a} = \sigma_a \circ G^m \circ \sigma_a^{-1}$.
\end{proof}

\begin{lemma}\label{comp2}
For each matching $a\in B^{n/2}$ the maps $\gamma_a$ and $\sigma_a \circ G^{n/2} \circ \sigma_a^{-1}$ from $S_a$ to $(S^2)^n$ are equal.
\end{lemma}

\begin{proof}
The map $\sigma_a^{-1}$ sends every arc of $a$ to an arc of the form $(2k-1,2k)$.  The map $G^{n/2}$ applies $g$ to the coordinates corresponding to each of these arcs and $\sigma_a$ returns them to their original positions.  Hence the composition is the identity on coordinates corresponding to right endpoints and is the antipodal map on coordinates corresponding to left endpoints. In other words $\gamma_a = \sigma_a^{-1} \circ G^{n/2} \circ \sigma_a$.
\end{proof}

\begin{corollary}\label{finally}
The image $\gamma_*(H_*(X_n))$ is the vector subspace of $H_*((S^2)^n)$ generated by $\bigcup_{a} \left((\gamma_a )_*(H_*(S_a))\right)$.
\end{corollary}
\begin{proof}
Lemmas \ref{comp1} and \ref{comp2} proved $\gamma_a = \sigma_a \circ G^{n/2} \circ \sigma_a^{-1}$ and $\gamma|_{S_a} = \sigma_a \circ G^m \circ \sigma_a^{-1}$.  Lemma \ref{Fmaps} showed that $G^m_*(q) = \pm G^{n/2}_*(q)$ for each homology generator of $H_*((X_2)^{n/2})$. This means
$$(\gamma|_{S_a})_*(q) = (\sigma_a  \circ \pm G^{n/2} \circ \sigma_a^{-1})_* (q) = \pm (\gamma_a)_*(q)$$
for each homology generator $q \in H_*(S_a)$.  As a result the complex vector subspace  $\gamma_*(H_*(X_n))$ in $H_*((S^2)^n)$ equals the subspace of $H_*((S^2)^n)$ spanned by $\bigcup_{a} \left((\gamma_a )_*(H_*(S_a))\right)$ . 
\end{proof}

\subsection{The $S_n$ action on $H_*(X_n)$}

There is a natural $S_n$ action on $(S^2)^n$ that permutes coordinates, in which $\sigma \in S_n$ acts on $(x_1, \ldots, x_n) \in (S^2)^n$ by $\sigma \cdot (x_1, \ldots, x_n) = (x_{\sigma^{-1}(1)}, \ldots, x_{\sigma^{-1}(n)})$.  This section proves that the induced $S_n$ action on $H_*((S^2)^n)$ restricts to the image $\gamma_*(H_*(X_n))$.  The proof computes the action of each simple transposition $s_i$ on each homology generator of $H_*(X_n)$ and shows it is a linear combination of elements of $H_*(X_n)$.  

We note that this is not true for most embeddings $X_n \rightarrow (S^2)^n$.  In fact similar computations  to Lemma \ref{fgmap} show it fails for the identity embedding even when $n=4$.  

Both $\gamma_*$ and the collection $\{(\gamma_a)_*\}$ map to the same vector subspace of $H_*((S^2)^n)$, but they  may differ by a sign on homology generators. To establish a sign convention, we choose a set of generators of $\gamma_*(H_*(X_n))$ as follows.
\begin{definition}
 Given a standard dotted noncrossing matching $M \in H_*(X_n)$ and the component $S_a$ for which $M\in H_*(S_a)$, define the generator of $\gamma_*(H_*(X_n))$ corresponding to $M$ to be $(\gamma_a)_*(M)$.
\end{definition}
The generator $(\gamma_a)_*(M)$ has many useful combinatorial properties which the following definitions will help make explicit.
\begin{definition}
The undot set of a line diagram is the set of all vertices which have undotted line segments. For example, the line diagrams in Figure \ref{lndg} have undot sets $\{ 2, 4 \}$ and $\{ 1,2,3 \}$.
\end{definition}  
 
The undot sets provide a bijective correspondence between line diagrams on $n$ nodes and subsets of $\{ 1, \ldots , n \}$. If $U$ is such a subset, denote by $l_U$ the line diagram with undot set $U$. 

Each standard dotted noncrossing matching determines a {\em line diagram sum}.
\begin{definition}
Let $M$ be a standard dotted noncrossing matching on $n$ vertices. Define $\mathcal{U}_M$ to be the collection of undot sets given by
\[\mathcal{U}_M = \{U \subseteq \{1,2,\ldots, n\}: U \textup{ contains exactly one endpoint of each undotted arc in }M\}.\]

For each set $U \in \mathcal{U}_M$, let $\Lambda_M(U)$ be the number of left endpoints in $U$.
Define the line diagram sum $L_M$ of $M$ to be
$$L_M = \sum_{U\in \mathcal{U}_M} (-1)^{\Lambda_M(U)} l_U.$$ 
\end{definition}

We give an example of $\mathcal{U}_M$ and $L_M$ for $M=$ \includegraphics[width=.6in]{unnest2.pdf}. The matching $M$ has undotted arc $(1,2)$ so $\mathcal{U}_M = \{ \{1\}, \{2 \} \}$ and $L_M$ is  
 $\begin{picture}(26, 9) (0,0)
\put(5,-5){\line(0,1){20}}
\put(5, 5){\circle*{5}}
\put(11,-5){\line(0,1){20}}
\put(17,-5){\line(0,1){20}}
\put(17, 5){\circle*{5}}
\put(23,-5){\line(0,1){20}}
\put(23,5){\circle*{5}}
\end{picture} - 
 \begin{picture}(26, 12) (0,0)
\put(0,-5){\line(0,1){20}}
\put(6, 5){\circle*{5}}
\put(6,-5){\line(0,1){20}}
\put(12,-5){\line(0,1){20}}
\put(12, 5){\circle*{5}}
\put(18,-5){\line(0,1){20}}
\put(18,5){\circle*{5}}
\end{picture}.$

\begin{lemma} \label{addarc}
Suppose the standard dotted noncrossing matching $M'$ is obtained from $M$ by inserting an arc $(i,j)$ with $1\leq i < j \leq n+2$. If $(i,j)$ is a dotted arc then
$$L_{M'} = \sum_{U\in \mathcal{U}_M} (-1)^{\Lambda_M(U)} l_{\widetilde{U_{\emptyset}}}$$ and if $(i,j)$ is an undotted arc then
$$L_{M'} = \sum_{U\in \mathcal{U}_M} (-1)^{\Lambda_M(U)} l_{\widetilde{U_j}} - \sum_{U\in \mathcal{U}_M} (-1)^{\Lambda_M(U)} l_{\widetilde{U_i}}$$ 
where if $S \subseteq \{i,j\}$ then $\widetilde{U}_S$ denotes the undot set obtained from $U$ by
\[\widetilde{U}_S = S \cup \{x \in U: x < i\} \cup \{x+1: x \in U \textup{ and } i \leq x < j \} \cup \{x+2: x \in U \textup{ and } j \leq x \}.\]
\end{lemma}
\begin{proof}
Assume $M'$ is obtained from $M$ by inserting the arc $(i,j)$ for $1\leq i < j \leq n+2$. Define the function $c: \{1, \ldots , n\} \rightarrow \{ 1, \ldots , n+2 \}$ by
$$c(x) =   \begin{cases}
x \text{ if } x<i, \\
x+1 \text{ if } i\leq x < j, \textup{ and } \\
x+2 \text{ if } j\leq x.
\end{cases}$$ 

First assume $(i,j)$ is a dotted arc. No additional undotted arcs have been added to the matching $M'$ so $|\mathcal{U}_M| = |\mathcal{U}_{M'}|$. Let $U\in \mathcal{U}_M$ and let $\widetilde{U_{\emptyset}}= \{ c(x) : x\in U \}$.  The line diagram $l_{\widetilde{U_{\emptyset}}}$ is obtained from $l_U$ by inserting one dotted segment between vertices $i-1$ and $i$ and another between vertices $j-1$ and $j$. 
For each $U\in \mathcal{U}_{M}$ we have $\Lambda_M(U) = \Lambda_{M'}(\widetilde{U_{\emptyset}})$.
We find
$$L_{M'} = \sum_{U\in \mathcal{U}_M} (-1)^{\Lambda_M(U)} l_{\widetilde{U_{\emptyset}}}.$$

Now assume $(i,j)$ is undotted. One undotted arc has been added to $M'$ so $|\mathcal{U}_{M'}| = 2\cdot |\mathcal{U}_{M}|$.  For each $U\in \mathcal{U}_M$ we construct two distinct undot sets $\widetilde{U_i}, \widetilde{U_j} \in \mathcal{U}_{M'}$. Let $\widetilde{U_i} =\widetilde{U_{\emptyset}} \cup \{ i \}$ and $\widetilde{U_j} = \widetilde{U_{\emptyset}} \cup \{ j \}$.   The line diagram $l_{\widetilde{U_i}}$ is obtained from $l_U$ by inserting an undotted segment between vertices $i-1$ and $i$ and a dotted segment between vertices $j-1$ and $j$; the diagram $l_{\widetilde{U_j}}$ reverses the roles of $i$ and $j$. Hence $\Lambda_M(U) = \Lambda_{M'}(\widetilde{U_j})$ and $\Lambda_M(U) + 1 = \Lambda_{M'}(\widetilde{U_i})$ and
$$L_{M'} = \sum_{U\in \mathcal{U}_M} (-1)^{\Lambda_M(U)} l_{\widetilde{U_j}} - \sum_{U\in \mathcal{U}_M} (-1)^{\Lambda_M(U)} l_{\widetilde{U_i}}.$$
 \end{proof}

 \begin{lemma}
 If $M\in H_*(S_a)$ is a standard dotted noncrossing matching $(\gamma_a)_*(M) = L_M$. 
 \end{lemma}
 
 \begin{proof}
 Lemma \ref{fgmap} proves the case when $n=2$.
 
 Assume as the inductive hypothesis that if $M$ is any standard dotted noncrossing matching on $n$ vertices then $(\gamma_a)_*(M) = L_M$. Let $a'$ be a matching obtained from $a$ by inserting an arc $(i,j)$ with $1\leq i<j\leq n+2$.  By construction the map $(\gamma_{a'})_*$ is a permutation of $(\gamma_a)_* \otimes g_*$ so that $g_*$ is applied to coordinates $i$ and $j$ and $(\gamma_a)_*$ is applied to the other coordinates (in their original order). 
 
 Let $M$ be a dotting of $a$ and construct a dotted noncrossing matching $M'$ by adding the dotted arc $(i,j)$. From the $n=2$ case  we know  $g_*([p]) = [p\times p]$. The homology generator $[p] \in H_*(S^2)$ is denoted by a dotted vertical line segment so the image $(\gamma_{a'})_*(M')$ agrees with $(\gamma_a)_*(M)$ except that dotted segments are inserted into the $i^{th}$ and $j^{th}$ positions. This is exactly the change to the line diagram sum described in Lemma \ref{addarc} as $l_{\widetilde{U_{\emptyset}}}$.
 
Now let $M'$ be obtained from $M$ by inserting an undotted arc $(i,j)$.  We know $g_*([c]) = [p\times c] - [c\times p]$. Hence each term in the image of $M$ appears twice in the image of $M'$: one occurrence of  $M$ has an undotted $i^{th}$ segment and dotted $j^{th}$ segment with sign changed; 
the other has a dotted $i^{th}$ segment and undotted $j^{th}$ segment with no change in sign.
This is exactly the formula described in Lemma \ref{addarc}. 
 \end{proof}
 
 \begin{corollary}
 The map $\gamma_*: H_*(X_n) \rightarrow H_*((S^2)^n)$ is injective.
 \end{corollary}
 
 \begin{proof}
Let $U_M$ be the undot set obtained from $M$ by choosing the right endpoint of each undotted arc.  Create a Young tableau from $U_M$ by putting the numbers from $U_M$ on the bottom row and all other numbers on the top row (with both rows in increasing order).  Theorem \ref{bijection between tableaux and matchings} showed that the standard noncrossing matchings $M$ with exactly $k$ undotted arcs are bijective with the tableaux associated to $M$, and hence so are the sets $U_M$.

Order subsets of $\{1,2,\ldots, n\}$ of cardinality $k$ according to the rule that $S= \{i_1 < i_2 < \ldots < i_k\}$ is less than $S'=\{i_1' < i_2' < \ldots < i_k'\}$ if for some $r$ both $i_r < i_{r}'$ and $i_{r+1}=i_{r+1}'$, $i_{r+2}=i_{r+2}'$, $\ldots$, $i_k = i_k'$.  Note that $U_M$ is the maximum set in $\mathcal{U}_M$ with respect to this order.

If the basis vectors $l_U \in H_*((S^2))^n)$ are ordered according to their undot sets then the map $M \mapsto L_M$ is represented by a matrix in row-echelon form; the leading term in each row is the coefficient of $l_{U_M}$, namely $\Lambda_M(U_M) = (-1)^0 = 1$.  Hence $\{L_M\}$ is linearly independent in $H_*((S^2)^n)$.  The standard dotted noncrossing matchings $M$ generate $H_*(X_n)$ so $\gamma_*$ is injective.
\end{proof}

Because homology generators for $H_*((S^2)^n)$ consist of a choice of point or two-cell for each coordinate, the $S_n$ action permutes the choices of point or two-cell in the same way it permutes coordinates. On the level of the line diagrams, this action permutes line segments. Figure \ref{permex} has an example.
\begin{figure}[h]
$$(1 2 3) \cdot 
\begin{picture}(26, 12) (0,0)
\put(5,-5){\line(0,1){20}}
\put(11, 5){\circle*{5}}
\put(11,-5){\line(0,1){20}}
\put(17,-5){\line(0,1){20}}
\put(17, 5){\circle*{5}}
\put(23,-5){\line(0,1){20}}
\put(23,5){\circle*{5}}
\end{picture}  = 
\begin{picture}(26, 12) (0,0)
\put(5,-5){\line(0,1){20}}
\put(5, 5){\circle*{5}}
\put(11,-5){\line(0,1){20}}
\put(17,-5){\line(0,1){20}}
\put(17, 5){\circle*{5}}
\put(23,-5){\line(0,1){20}}
\put(23,5){\circle*{5}}
\end{picture}$$
\caption{Permutation action in $H_*((S^2)^4)$.}\label{permex}
\end{figure}

Let $s_i$ be the simple transposition exchanging $i$ and $i+1$. Choose a homology generator $M\in H_*(S_a)$. Since $\gamma_*$ is injective and $\gamma_*(M) = \pm (\gamma_a)_*(M)$, we may identify the image $(\gamma_a)_*(M)$ with $M$.  We will show that $s_i (M) \in H_*(X_n)$, namely the $S_n$ action on $H_*((S^2)^n)$ restricts to $\gamma_*(H_*(X_n))$.  

\begin{theorem} \label{well-defined representation}
The chart in Figure \ref{snact} describes the action of simple transpositions on $\gamma_*(H_*(X_n))$.  The $S_n$ - action on $H_*(X_n)$ defined so that $\sigma \in S_n$ acts on a generator $M$ by
\[\sigma \cdot M := \sigma \cdot L_M\]
is a well-defined representation.
\end{theorem}

\begin{proof}
The action of $S_n$ on $(S^2)^n$ is consistent with the $S_n$-actions on sets and line diagrams given as follows.  If $U$ is an undot set then $\sigma \in S_n$ acts on $U$ by $\sigma \cdot U = \{ \sigma \cdot x : x\in U \}$ and $\sigma \cdot \mathcal{U}_M = \{ \sigma \cdot U \textup{ for } U\in \mathcal{U}_M\}$.  By definition $\sigma \cdot l_U = l_{\sigma \cdot U}$ and 
$$ \sigma \cdot L_M = \sum_{U \in \mathcal{U}_M} (-1)^{\Lambda_M(U)}l_{\sigma \cdot U}.$$ 
The simple transpositions $s_i$ generate the symmetric group; to prove that $S_n$ restricts to an action on $\gamma_*(H_*(X_n))$ it suffices to show that for each line diagram sum $L_M$ the image $s_i (L_M)$ is a linear combination of line diagram sums. We do this by explicitly calculating $s_i(L_M)$ in the four possible configurations of arcs and dottings: if $i$ and $i+1$ are both incident to dotted arcs in $M$ (whether or not $(i,i+1)$ is an arc); if $(i,i+1)$ is an undotted arc in $M$; if $i$ and $i+1$ are incident to different arcs in $M$, exactly one of which is dotted; and if $i$ and $i+1$ are incident to different arcs in $M$, neither of which is dotted.  Since the map $M \mapsto L_M$ is injective, the preimage of each $s_i(L_M)$ is unique and the $S_n$-action on $H_*(X_n)$ is a well-defined $S_n$-representation.

Figure \ref{snact} gives each case diagrammatically. As in Lemma \ref{nestcases}, the chart only shows arcs incident to vertices $i$ and $i+1$, since any other arcs are common to $M$ and $M'$.

\noindent {\bf Case 1}: Every arc in the matching $M$ incident to $i$ and $i+1$ is dotted.

This case encompasses both the case when $(i, i+1)$ is a dotted arc in $M$ and the case when $(j,i)$ and $(i+1,k)$ are both dotted arcs in $M$.  For all such $M$ no $U \in \mathcal{U}_M$ contains either $i$ or $i+1$.  Hence $s_i \cdot U = U$ for all $U \in \mathcal{U}_M$.  By definition $s_i \cdot L_M = L_M$ in this case.

\noindent {\bf Case 2}: The matching $M$ contains the undotted arc $(i,i+1)$.

If $(i, i+1)$ is an undotted arc then let $\mathcal{U}_i$ be the collection of all undot sets in $\mathcal{U}_M$ containing $i$ and let $\mathcal{U}_{i+1}$ be the undot sets containing $i+1$. By definition $\mathcal{U}_i$ and $\mathcal{U}_{i+1}$ are disjoint and $\mathcal{U}_M = \mathcal{U}_i \cup \mathcal{U}_{i+1}$.  The map $s_i: \mathcal{U}_i \rightarrow \mathcal{U}_{i+1}$ is a bijection with inverse $s_i$.  In particular $s_i \cdot \mathcal{U}_M = \mathcal{U}_M$. The vertex $i$ is a left endpoint and $i+1$ is a right endpoint so 
\[(-1)^{\Lambda_M(U)-1} = (-1)^{\Lambda_M(s_i\cdot U)}\]
if $U \in \mathcal{U}_M.$ 
It follows that 
\[ \begin{array}{rcl} s_i\cdot L_M & = & \sum_{U\in \mathcal{U}_i} (-1)^{\Lambda_M(U)}l_{s_i\cdot U} + \sum_{U\in \mathcal{U}_{i+1}} (-1)^{\Lambda_M(U)}l_{s_i\cdot U} \\
& = &  \sum_{U\in \mathcal{U}_i} (-1)^{\Lambda_M(U)-1}l_{U} + \sum_{U\in \mathcal{U}_{i+1}} (-1)^{\Lambda_M(U)-1}l_{U} = -L_M.
\end{array}\]

\noindent {\bf Note}: Both Case 3 and Case 4 assume the standard matching $M$ contains arcs $(j,i)$ and $(i+1,k)$ with no additional hypotheses on the positions of $j$ and $k$. All three possible relative positions of the arcs $(j,i)$ and $(i+1,k)$, shown in Figure \ref{nestcases}, are consistent with our calculations.
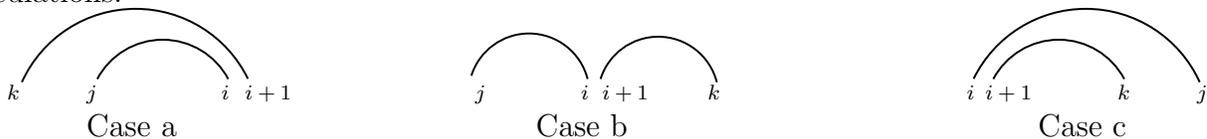
\begin{figure}[h]
\begin{picture} (100, 38)(0,-10)
\put(-160,10){\includegraphics[width=1.2in]{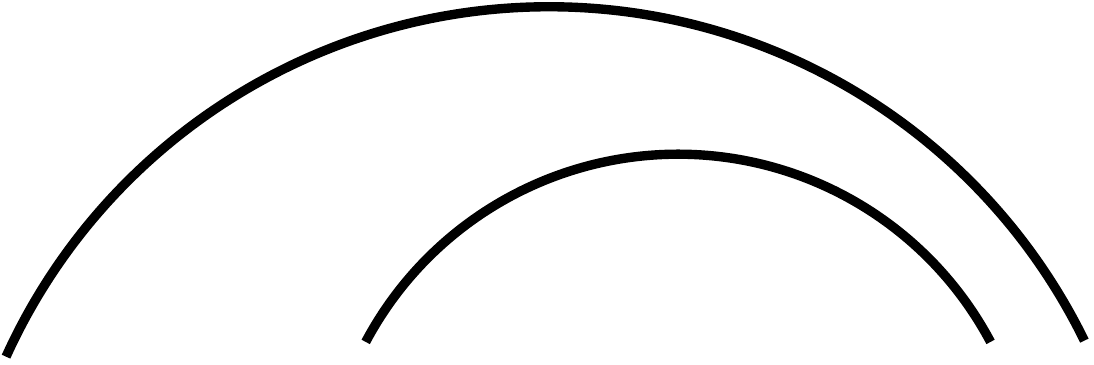}}
\put(10,10){\includegraphics[width=1.3in]{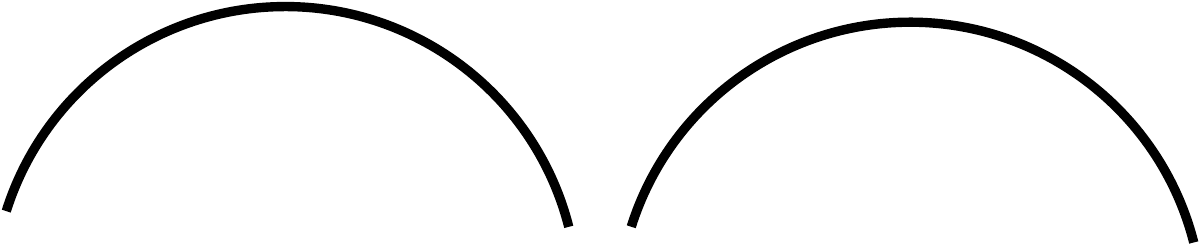}}
\put(200,10){\includegraphics[width=1.2in]{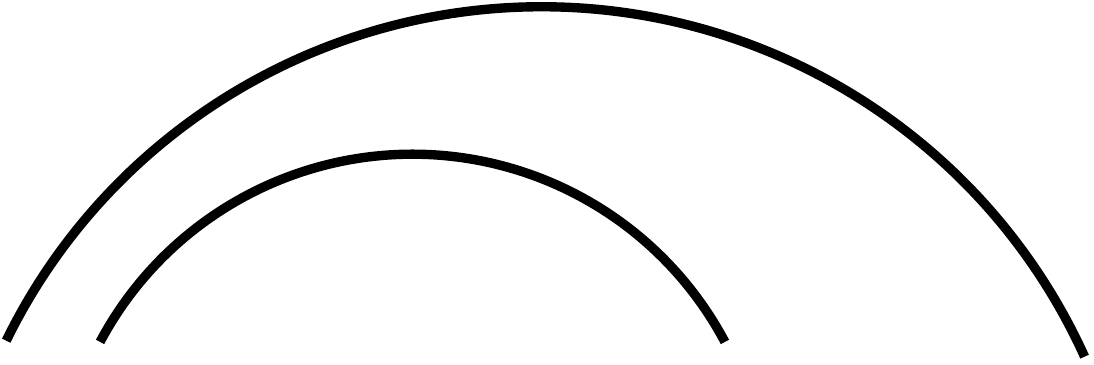}}

\put(-165,4){\tiny{$k$}}
\put(-135,4){\tiny{$j$}}
\put(-84,4){\tiny{$i$}}
\put(-75,4){\tiny{$i+1$}}

\put(12,4){\tiny{$j$}}
\put(52,4){\tiny{$i$}}
\put(60,4){\tiny{$i+1$}}
\put(100,4){\tiny{$k$}}

\put(198,4){\tiny{$i$}}
\put(205,4){\tiny{$i+1$}}
\put(255,4){\tiny{$k$}}
\put(285,4){\tiny{$j$}}

\put(-135,-10){Case a}
\put(35,-10){Case b}
\put(225,-10){Case c}
\end{picture}
\caption{Arcs incident on vertices $i$ and $i+1$.  (We allow additional arcs whose endpoints are consistent with those in the diagram.)} \label{nestcases}
\end{figure}

\noindent {\bf Case 3}: The standard matching $M$ contains the arcs $(j,i)$ and $(i+1,k)$, exactly one of which is dotted. (There are no additional hypotheses on $j$ and $k$.)

Assume arc $(j, i)$ is undotted and arc $(i+1,k)$ is dotted; the case when $(j,i)$ is dotted is symmetric.  Let $\mathcal{U}_j$ be the undot sets in $\mathcal{U}_M$ containing $j$ and let $\mathcal{U}_{i}$ be the undot sets containing $i$. As in Case 1 there is a disjoint union $\mathcal{U}_M = \mathcal{U}_j \cup \mathcal{U}_i$. Let $M'$ be the standard matching with dotted arc $(j,k)$ and undotted arc $(i,i+1)$ that is otherwise identical to $M$. Then $\mathcal{U}_{M'} = \mathcal{U'}_i \cup \mathcal{U'}_{i+1}$. 

Matchings $M$ and $M'$ each have an undotted arc incident to vertex $i$; the other undotted arcs in $M$ are exactly the same as those in $M'$. Thus $\mathcal{U}_i = \mathcal{U'}_i$.  The matching $M$ is standard so $j < i$.  For each $U \in \mathcal{U}_i$  we have $\Lambda_M(U) = \Lambda_{M'}(U)-1$ since $i$ is a right endpoint in $M$ and a left endpoint in $M'$. The maps $s_i: \mathcal{U}_j \rightarrow \mathcal{U}_j$ and $s_i: \mathcal{U}_i \rightarrow \mathcal{U'}_{i+1}$ are bijections. For each $U\in \mathcal{U}_i$ we have $\Lambda_M(U) = \Lambda_{M'}(s_i \cdot U)$ since $i$ is a right endpoint in $M$ and $i+1$ is a right endpoint in $M'$. Together this gives
\[\begin{array}{rcl}
s_i\cdot L_M &=& \displaystyle{\sum_{U\in \mathcal{U}_j} (-1)^{\Lambda_M(U)}l_{s_i\cdot U} + \sum_{U\in \mathcal{U}_{i}} (-1)^{\Lambda_M(U)}l_{s_i\cdot U}} \\ 
&=& \displaystyle{ \sum_{U\in \mathcal{U}_j} (-1)^{\Lambda_M(U)}l_{U} + \sum_{U\in \mathcal{U'}_{i+1}} (-1)^{\Lambda_{M'}(U)}l_{U} }\\ 
&=& \displaystyle{\sum_{U\in \mathcal{U}_j} (-1)^{\Lambda_M(U)}l_{U} + \sum_{U\in \mathcal{U'}_{i+1}} (-1)^{\Lambda_{M'}(U)}l_{U} +  \sum_{U\in \mathcal{U}_{i}} (-1)^{\Lambda_M(U)}l_{ U} - \sum_{U\in \mathcal{U}_{i}} (-1)^{\Lambda_M(U)}l_{ U} } \\ 
&=& \displaystyle{\sum_{U\in \mathcal{U}_j} (-1)^{\Lambda_M(U)}l_{U} + \sum_{U\in \mathcal{U}_{i}} (-1)^{\Lambda_M(U)}l_{U} + \sum_{U\in \mathcal{U'}_{i}} (-1)^{\Lambda_{M'}(U)}l_{ U} + \sum_{U\in \mathcal{U'}_{i+1}} (-1)^{\Lambda_{M'}(U)}l_{U}} \\ &=& L_M + L_{M'}. \end{array} \]

\noindent {\bf Case 4}: The standard matching $M$ contains the undotted arcs $(j,i)$ and $(i+1,k)$. (There are no additional hypotheses on $j$ and $k$.)

Suppose both $(j,i)$ and $(i+1,k)$ are undotted.  For each subset $S \subseteq \{j,i,i+1,k\}$ let $\mathcal{U}_S$ denote the collection of $U \in \mathcal{U}$ with $S \subseteq U$.  There is a disjoint union 
\[\mathcal{U} _M = \mathcal{U}_{j,i+1} \cup \mathcal{U}_{j,k} \cup \mathcal{U}_{i,i+1} \cup \mathcal{U}_{i,k}.\] 
Let $M'$ be the matching with undotted arcs $(j,k)$ and $(i,i+1)$ that is otherwise identical to $M$. Then $\mathcal{U}_{M'} = \mathcal{U'}_{j,i} \cup \mathcal{U'}_{j,i+1} \cup \mathcal{U'}_{i,k} \cup \mathcal{U'}_{k,i+1}$. The sets  $\mathcal{U}_{j,i+1} = \mathcal{U'}_{j,i+1}$ and $\mathcal{U}_{i,k} = \mathcal{U'}_{i,k}$ since $M$ and $M'$ agree off of vertices $j,k,i$, and $i+1$.  The maps $s_i: \mathcal{U}_{i,k} \rightarrow \mathcal{U'}_{k,i+1}$ and $s_i: \mathcal{U}_{j,i} \rightarrow \mathcal{U'}_{j,i+1}$ are bijections while 
$s_i: \mathcal{U}_{j,k} \rightarrow \mathcal{U}_{j,k}$ and $s_i: \mathcal{U}_{i,i+1} \rightarrow \mathcal{U}_{i,i+1}$ are identity maps.

In all cases shown in Figure \ref{nestcases}, for each $U\in \mathcal{U}_{j,i+1}$ we have $\Lambda_M(U) = \Lambda_{M'}(s_i\cdot U)$ and for each $U\in \mathcal{U}_{i,k}$ we have $\Lambda_M(U) = \Lambda_{M'}(s_i\cdot U)$.  If $U \in \mathcal{U}_{i,k}$ or $U \in \mathcal{U}_{j,i+1}$ then $s_i$ changes exactly one endpoint in $U$ from right to left or vice versa, namely $\Lambda_M(U) = \Lambda_{M'}(U) \pm 1$.  This gives 
\[\begin{array}{rcl}
s_i\cdot L_M &=& \displaystyle{\sum_{U\in \mathcal{U}_{j,i+1}} (-1)^{\Lambda_M(U)}l_{s_i\cdot U} + \sum_{U\in \mathcal{U}_{j,k}} (-1)^{\Lambda_M(U)}l_{s_i \cdot U} + \sum_{U\in \mathcal{U}_{i,i+1}} (-1)^{\Lambda_M(U)}l_{s_i\cdot U} + \sum_{U\in \mathcal{U}_{i, k}} (-1)^{\Lambda_M(U)}l_{s_i \cdot U}} \\ 
&=& \displaystyle{ \sum_{U\in \mathcal{U'}_{j,i}} (-1)^{\Lambda_{M'}(U)}l_{U} + \sum_{U\in \mathcal{U}_{j,k}} (-1)^{\Lambda_M(U)}l_{U} + \sum_{U\in \mathcal{U}_{i, i+1}} (-1)^{\Lambda_M(U)}l_{U} + \sum_{U\in \mathcal{U'}_{i+1,k}} (-1)^{\Lambda_{M'}(U)}l_{U} } \\ 
&& \hspace{.25in}  \displaystyle{+ \sum_{U\in \mathcal{U}_{j, i+1}} (-1)^{\Lambda_M(U)}l_{U} - \sum_{U\in \mathcal{U}_{j, i+1}} (-1)^{\Lambda_M(U)}l_{U} + \sum_{U\in \mathcal{U}_{i,k}} (-1)^{\Lambda_M(U)}l_{U} -  \sum_{U\in \mathcal{U}_{i,k}} (-1)^{\Lambda_M(U)}l_{U}} \\ 
&=& L_M + L_{M'}.\end{array}\]
\end{proof}

\begin{figure}[h]
\begin{tabular}{|c||c|c|}
\cline{1-3} & &   \hspace{1em} \\
 {\bf Vertex Labelings} & $M$ & $s_i\cdot M$  \\
& &   \hspace{1em} \\
 \cline{1-3} & &  \hspace{1em} \\
 {\bf Case 1} &\includegraphics[width=20pt]{x2dot.pdf}& \includegraphics[width=20pt]{x2dot.pdf}\\
 \cline{2-3} & & \hspace{1em} \\
 \raisebox{7pt}{\Small{$i,i+1$ both on dotted arcs}}&  \includegraphics[width=.6in]{unnest3.pdf} &  \includegraphics[width=.6in]{unnest3.pdf} \\
 \cline{1-3} & & \hspace{1em}\\
 {\bf  Case 2} & \includegraphics[width=.3in]{x2.pdf} &  $-$ \includegraphics[width=.3in]{x2.pdf} \\
 \raisebox{7pt}{\Small{$(i,i+1)$ is an undotted arc}} & & \\
 \cline{1-3} & & \hspace{1em}\\
  {\bf  Case 3}  &\includegraphics[width=.6in]{unnest1.pdf} & $\includegraphics[width=.6in]{unnest1.pdf} + \includegraphics[width=.6in]{nest2.pdf}$\\
\Small{$(j,i)$ and $(i+1,k)$ have one dot} & & \\
\cline{2-3} & & \hspace{1em}\\
 \Small{(the dotted arc is on right or left} & & \\
 \Small{depending on the sign of i-j)}&   \includegraphics[width=.6in]{nest2.pdf} & $\includegraphics[width=.6in]{nest2.pdf} + \includegraphics[width=.6in]{unnest1.pdf}$\\
 \cline{1-3} & & \hspace{1em} \\
 {\bf  Case 4} & \includegraphics[width=.6in]{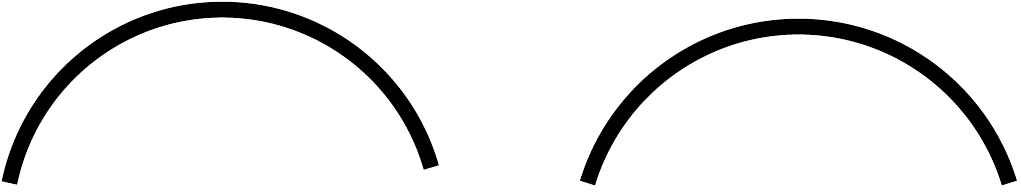}& $\includegraphics[width=.6in]{revunnest0.pdf} + \includegraphics[width=.6in]{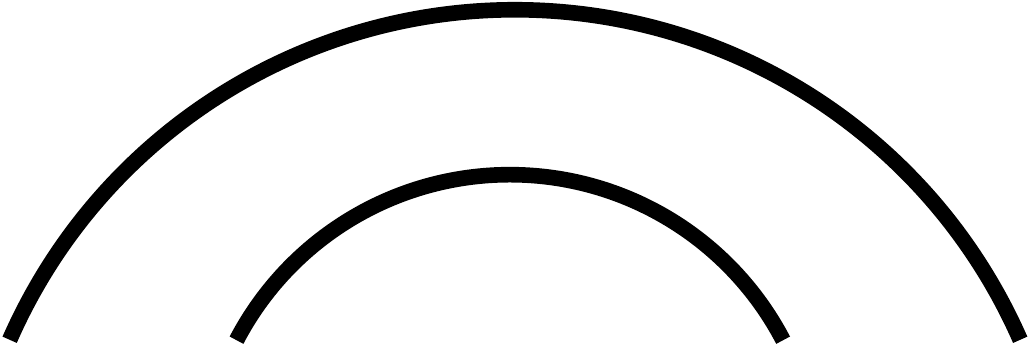}$ \\
 \cline{2-3} & & \hspace{1em}\\
\raisebox{7pt}{\Small{$(j,i)$ and $(i+1,k)$ have no dots}} & \includegraphics[width=.6in]{nest.pdf} &$\includegraphics[width=.6in]{nest.pdf} + \includegraphics[width=.6in]{revunnest0.pdf}$ \\

 \hline
 \end{tabular}
\caption{The $S_n$-action on standard dotted noncrossing matchings.}\label{snact}
\end{figure}

\section{The $S_n$ representation on $H_k(X_n)$}

This section explicitly identifies the $S_n$ representation on the homology of the Springer fiber in each degree.  First we review Specht modules, a classical construction of irreducible representations of $S_n$.  Then we give a direct proof that the $\C$-module $\gamma_*(H_k(X_n))$ with the $S_n$-action from the antipodal embedding is the Specht module of the partition $(n-k,k)$.

\subsection{Specht modules for partitions $(n-k,k)$}\label{classical repn theory}

We sketch the classical construction of Specht modules both in general and for two-row partitions; \cite[Chapter 7]{Fulton} has more.

Let $\lambda$ be a partition of $n$ considered as a Young diagram (a collection of left- and top-aligned boxes so that the $i^{th}$ row has $\lambda_i$ boxes for each $i$).  A Young tableau is a filling of the Young diagram $\lambda$ with the numbers $1,2,\ldots, n$ without repetition.  A {\em tabloid} $T$ is an equivalence class of Young tableaux, where the equivalence relation is reordering of the numbers in each row.  The permutation group $S_n$ acts on the collection of tabloids; the action of $w \in S_n$ on the tabloid $T$ replaces $i$ with $w(i)$ for each $i = 1, \ldots, n$.  For each {\em tableau} $T$ let $\textup{Col}(T)$ denote the subgroup of permutations in $S_n$ that stabilize each column of $T$.

Consider the complex vector space with one basis element $v_T$ for each tabloid $T$ of shape $\lambda$.  The Specht module $V_{\lambda}$ is the subspace spanned by the vectors 
\[e_T = \sum_{w \in \textup{Col}(T)} \textup{sign}(w) v_{w(T)}\]
for each {\em tableau} $T$ of shape $\lambda$.  Figure \ref{Specht generators} gives an example.  A classical theorem states that $V_{\lambda}$ is the unique (up to isomorphism) irreducible representation of $S_n$ corresponding to $\lambda$.  A small exercise shows that $V_{\lambda}$ has a basis $\{e_T\}$ indexed by {\em standard} Young tableaux, namely tableaux whose rows increase left-to-right and whose columns increase top-to-bottom.

The construction of the Specht module is simpler when the partition has two rows, as in Figure \ref{Specht generators}.  The subgroup $\textup{Col}(T)$ for a two-row standard Young tableau $T$ is a product of transpositions, one for each column of height two.  The sign of the permutation $w \in \textup{Col}(T)$ is determined by the parity of the number of columns in which $w$ exchanges entries.  

\begin{figure}[h]
\begin{tabular}{|c|c|}
\cline{1-2} & \hspace{1em} \\
Tableau $T$ & Basis element $e_T$ \\
 & \hspace{1em} \\
\cline{1-2} & \hspace{1em} \\
$\begin{array}{|c|c|} \cline{1-2} 1 & 3 \\ \cline{1-2} 2 & 4 \\ \hline \end{array}$ & 
$\begin{array}{|c|c|} \cline{1-2} 1 & 3 \\ \cline{1-2} 2 & 4 \\ \hline \end{array} - \begin{array}{|c|c|} \cline{1-2} 2 & 3 \\ \cline{1-2} 1 & 4 \\ \hline \end{array} - \begin{array}{|c|c|} \cline{1-2} 1 & 4 \\ \cline{1-2} 2 & 3 \\ \hline \end{array} + \begin{array}{|c|c|} \cline{1-2} 2 & 4 \\ \cline{1-2} 1 & 3 \\ \hline \end{array}$ \\
 & \hspace{1em} \\
\cline{1-2} & \hspace{1em} \\
$\begin{array}{|c|c|} \cline{1-2} 1 & 2 \\ \cline{1-2} 3 & 4 \\ \hline \end{array}$ & 
$\begin{array}{|c|c|} \cline{1-2} 1 & 2 \\ \cline{1-2} 3 & 4 \\ \hline \end{array} - \begin{array}{|c|c|} \cline{1-2} 2 & 3 \\ \cline{1-2} 1 & 4 \\ \hline \end{array} - \begin{array}{|c|c|} \cline{1-2} 1 & 4 \\ \cline{1-2} 2 & 3 \\ \hline \end{array} + \begin{array}{|c|c|} \cline{1-2} 3 & 4 \\ \cline{1-2} 1 & 2 \\ \hline \end{array}$ \\
 & \hspace{1em} \\
\hline \end{tabular}
\caption{Generators of the Specht module for the partition $(2,2)$}\label{Specht generators}
\end{figure}

\subsection{The $S_n$ action on the homology of $X_n$}

The standard noncrossing matchings naturally define another subspace $W_{\lambda}$ of the complex vector space spanned by 
\[\{v_T: T \textup{ is a tabloid of shape } (n-k,k)\}.\]  
We define a set of generators for $W_{\lambda}$ and then show that in fact $V_{\lambda} = W_{\lambda}$.

\begin{definition}
Let $M$ be a standard noncrossing matching with exactly $k$ undotted arcs and let $T = \varphi(M)$.  Define the subgroup $\textup{Undot}(M) \subseteq S_n$ to be the collection of permutations generated by the transpositions $(ij)$ for each undotted arc $(i,j)$.  The matching module $W_{\lambda}$ is the subspace spanned by the vectors
\[e_M = \sum_{w \in \textup{Undot}(M)} \textup{sign}(w) v_{w(T)}.\]
\end{definition}

Figure \ref{matching generators} gives these generators for the standard noncrossing undotted matchings when $n=4$.  Note that the second generator in Figure \ref{matching generators} is not the same as in Figure \ref{Specht generators}.
\begin{figure}[h]
\begin{tabular}{|c|c|c|}
\cline{1-3}& & \hspace{1em} \\
Matching $M$ & Tableau $\varphi(M)$ & Basis element $e_M$ \\
&  & \hspace{1em} \\
 \cline{1-3} &  & \hspace{1em} \\
\raisebox{-2pt}{\includegraphics[width = .8in]{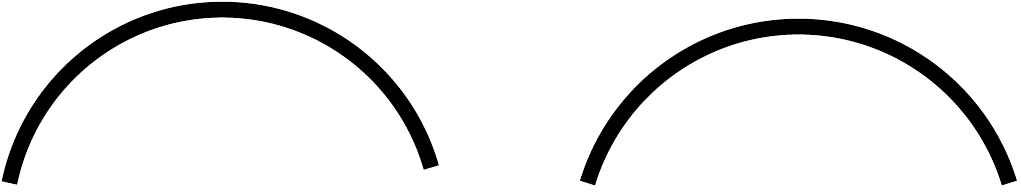}} & $\begin{array}{|c|c|} \cline{1-2} 1 & 3 \\ \cline{1-2} 2 & 4 \\ \hline \end{array}$ & 
$\begin{array}{|c|c|} \cline{1-2} 1 & 3 \\ \cline{1-2} 2 & 4 \\ \hline \end{array} - \begin{array}{|c|c|} \cline{1-2} 2 & 3 \\ \cline{1-2} 1 & 4 \\ \hline \end{array} - \begin{array}{|c|c|} \cline{1-2} 1 & 4 \\ \cline{1-2} 2 & 3 \\ \hline \end{array} + \begin{array}{|c|c|} \cline{1-2} 2 & 4 \\ \cline{1-2} 1 & 3 \\ \hline \end{array}$ \\
& & \hspace{1em} \\
\cline{1-3} && \hspace{1em} \\
\raisebox{-5pt}{\includegraphics[width=.8in]{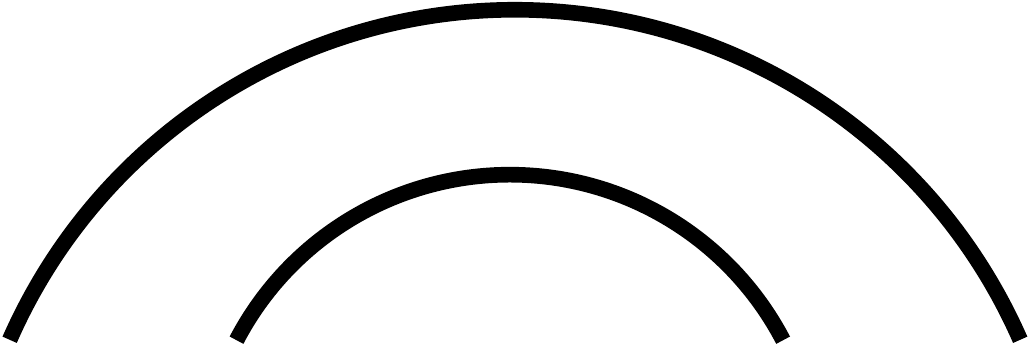}} &$\begin{array}{|c|c|} \cline{1-2} 1 & 2 \\ \cline{1-2} 3 & 4 \\ \hline \end{array}$ & 
$\begin{array}{|c|c|} \cline{1-2} 1 & 2 \\ \cline{1-2} 3 & 4 \\ \hline \end{array} - \begin{array}{|c|c|} \cline{1-2} 2 & 4 \\ \cline{1-2} 1 & 3 \\ \hline \end{array} - \begin{array}{|c|c|} \cline{1-2} 1 & 3 \\ \cline{1-2} 2 & 4 \\ \hline \end{array} + \begin{array}{|c|c|} \cline{1-2} 3 & 4 \\ \cline{1-2} 1 & 2 \\ \hline \end{array}$ \\
& & \hspace{1em} \\
\hline \end{tabular}
\caption{Generators of the matching module for the partition $(2,2)$}\label{matching generators}
\end{figure}

\begin{lemma}
The map $M \mapsto e_M$ induces an $S_n$-equivariant isomorphism of complex vector spaces $H_{2k}(X_n) \cong W_{(n-k,k)}$.
\end{lemma}

\begin{proof}
We first define an $S_n$-equivariant $\mathbb{C}$-linear isomorphism $\psi$ that sends line diagrams to tabloids. We then show $M \mapsto e_M$ factors $S_n$-equivariantly through $\psi$.  

Let $l$ be the line diagram for a basis element of $H_*((S^2)^n)$.  Assign the line diagram $l$ to the tabloid $\psi(l)$ which has $i$ on the top row if the $i^{th}$ line is dotted and on the second row otherwise.  Line diagrams are in bijection with their undot sets, and $k$-element undot sets are in bijection with second rows of tabloids for the partition $(n-k,k)$. Thus $\psi$ is a bijection between the basis for $H_{2k}((S^2)^n)$ and the basis for the space of tabloids for the partition $(n-k,k)$. 

Extending by $\C$-linearity gives a map from $H_{2k}((S^2)^n)$ to the vector space of tabloids for the partition $(n-k,k).$ The permutation $w$ sends the $i^{th}$ strand of $l$ to the $w(i)^{th}$ position.  Hence $\psi(w(l))$ has $w(i)$ on the top row if and only if the $i$ is on the top row in $\psi(l)$ so the map $\psi$ is $S_n$-equivariant.

We show that if $L_M$ is the line diagram for $M$ then $\psi(L_M) = e_M$.  If $M$ has a dotted arc $(i,j)$ then both $i$ and $j$ are dotted lines in every term of $L_M$, and both $i$ and $j$ are on the top row of every tabloid $w(\varphi(M))$ for $w \in \textup{Undot}(M)$.  

If $M$ has no undotted arcs, the line diagram sum $L_M$ has a single term with all $n$ lines dotted. The group $\textup{Undot}(M)$ is trivial so $e_M$ is the standard tableaux with a single row of $n$ boxes. Therefore $\psi(L_M) = e_M$ and the base case is proven. 

Assume as the inductive hypothesis that $\psi(L_M)=e_M$ if $M$ has at most $k-1$ undotted arcs.  Suppose $M'$ is obtained by inserting an undotted arc $(i,j)$ into $M$, renumbering vertices as needed.   The line diagram $L_{M'}$ has two terms for each term $l$ of $L_M$.  In the first term, the $i^{th}$ strand is undotted, the $j^{th}$ strand is dotted, and the sign differs with that of $l$. In the second term, the $i^{th}$ strand is dotted, the $j^{th}$ strand is undotted, and the sign agrees with that of $l$.   The tableau $\varphi(M')$ agrees with $\varphi(M)$ except that $i$ is inserted into the top row, $j$ into the bottom row, and all other entries are renumbered accordingly.  

Given $w\in \textup{Undot}(M)$, let $w'\in \textup{Undot}(M')$ be the permutation obtained by renumbering to allow for the insertion of the arc $(i,j)$. Then $e_{M'}$ has two terms for each term $v_{w(\varphi(M))}$ of $e_{M}$, namely $\textup{sign}(w') v_{w'(\varphi(M'))}$ and $\textup{sign}(s_{ij}w') v_{s_{ij}w'(\varphi(M'))}$.  These are exactly the images under $\psi$ of the new terms in the line diagram.  
By induction on the number of undotted arcs we conclude $\psi(L_M) = e_M$ for each standard noncrossing matching $M$.

We conclude that $\psi$ restricts to an isomorphism of $\gamma_*(H_{2k}(X_n))$ with $W_{(n-k,k)}$. The map $M\mapsto L_M$ is injective so the map $M\mapsto e_M$ is an isomorphism of $H_{2k}(X_n)$ with $W_{(n-k,k)}$. The $S_n$-action on matchings is defined so that the map $M \mapsto L_M$  is $S_n$-equivariant; since $\psi$ is $S_n$-equivariant, the map $M\mapsto e_M$ is an $S_n$- equivariant isomorphism.  
\end{proof}

\begin{theorem} \label{springer repn}
Let $V_{\lambda}$ denote the Specht module corresponding to $\lambda$.  Let $W_{\lambda}$ denote the module defined by the standard noncrossing matchings.  Then $V_{\lambda}$ equals $W_{\lambda}$ as $S_n$-representations.  In particular the $S_n$ action on $H_*(X_n)$ is the Springer representation.
\end{theorem}

\begin{proof} Both $V_{\lambda}$ and $W_{\lambda}$ are $S_n$-representations in the vector space of tabloids.  This means that their intersection (in the vector space generated by tabloids) is also an $S_n$ representation.  
We know that $V_{\lambda}$ is an irreducible representation, so the only subspaces of $V_{\lambda}$ that are preserved by the $S_n$ action are $0$ and $V_{\lambda}$.  Hence $V_{\lambda} \cap W_{\lambda}$ is either $0$ or $V$.    
                                                                               
Let $M_0$ be the unnested standard noncrossing matching with $n-k$ undotted arcs on the left and $k$ dotted arcs on the right.  Let $T_0$ be the standard tableau with $2,4,6,8,10,\ldots, 2n-2k$ in the second row.  The tableau $T_0$ corresponds to $M_0$ under the bijection of Section \ref{Springer topology}.  In this case the columns are exactly the undotted arcs and so $e_{M_0} = e_{T_0}$.

This shows that $V_{\lambda} \cap W_{\lambda}$ contains at least one nonzero vector, and hence is $V_{\lambda}$.  Since $V_{\lambda}$ and $W_{\lambda}$ have the same dimension, we conclude that $V_{\lambda} = V_{\lambda} \cap W_{\lambda} = W_{\lambda}$.  By Proposition \ref{GP identification of Springer rep} this is the Springer representation.
\end{proof}

\begin{conjecture}
The Kazhdan-Lusztig basis for irreducible representations of $S_n$ is defined to be the geometric basis induced by the representation on the top-degree cohomology of the Springer fiber.  This is the basis we identified in top degree.  We conjecture that in fact the matching basis is the Kazhdan-Lusztig basis in all degrees.
\end{conjecture}

\section{appendix}

We review the relevant results from \cite{Cautis-Kamnitzer} and prove that Khovanov's construction of the Springer variety is diffeomorphic to the classical construction of the Springer variety.

Our notation uses that of \cite{Cautis-Kamnitzer} except where it conflicts with notation already used in this paper.  Let $X$ be a linear operator $X: \C^{2N} \rightarrow \C^{2N}$ with Jordan blocks of size $(N,N)$.  Fix a basis $e_1, \ldots, e_N, f_1, \ldots, f_N$ for $\C^{2N}$ so that $X (e_i)=e_{i-1}$ and $X(f_i) = f_{i-1}$ for each $i \geq 2$, and so that $\ker X$ is the span of $e_1, f_1$.  We use the Hermitian inner product with orthonormal basis $e_1, e_2, \ldots, e_N, f_1, f_2, \ldots, f_N$.  Assume that $N \geq 2n$ and define a partial Springer variety of nested vector subspaces of $\C^{2N}$ by
\[Y_n = \{(V_1 \subseteq V_2 \subseteq \cdots \subseteq V_n): \textup{ each } V_i \textup{ is $i$-dimensional and } XV_i \subseteq V_{i-1}\}.\]
Recall that $B^{n/2}$ is the set of all noncrossing matchings on $n$ vertices. Define
\[X'_n = \bigcup_{a \in B^{n/2}} S'_{a,n}\]
where 
\[S'_{a,n} = \{(x_1, x_2, \ldots, x_n) \in (S^2)^n: x_i = -x_j \textup{ if } (i,j) \in a\}.\]
(As before $-x$ denotes the image of $x$ under the antipodal map.)  

\cite[Theorem 2.1]{Cautis-Kamnitzer} proves
\begin{proposition} {\bf (Cautis-Kamnitzer)}  Choose an orthonormal basis $e,f$ for $\mathbb{P}^1$ and define a map $C: \mathbb{P}^{2N-1} \rightarrow \mathbb{P}^1$ so that
$$\textup{span}_{\C} \left( \sum a_i e_i + \sum b_i f_i \right) \mapsto \textup{span}_{\C}\left( \left(\sum a_i\right) e + \left( \sum b_i\right)  f\right) .$$  For each $V_{\bullet} \in Y_n$ define lines $L_1, \ldots, L_n$ by $V_i = V_{i-1} \oplus L_i$ and $L_i \perp V_{i-1}$.  Then the map $\ell: Y_n \rightarrow (\mathbb{P}^1)^n$ defined by
\[V_{\bullet} \mapsto (C(L_1), C(L_2), \ldots, C(L_n))\]
is a diffeomorphism.

Furthermore define $Z_n^i = \{V_{\bullet} \in Y_n: XV_{i+1} = V_{i-1}\}$.  Then the image of $Z_n^i$ under the diffeomorphism $\ell$ is exactly the elements in $(\mathbb{P}^1)^n$ 
satisfying $C(L_i) = -C(L_{i+1})$.
\end{proposition}

For each noncrossing matching $a$ and each $n$ let $C_{a,n}$ denote the preimage $\ell^{-1}(S'_{a,n})$.

Let $a$ be a noncrossing matching on $n$ vertices containing the arc $(i,i+1)$ and let $a'$ be a noncrossing matching on $n-2$ vertices. Suppose $a'$ is obtained from $a$ by erasing the arc $(i,i+1)$, so that if $f: \{1,2,\ldots, n\} \rightarrow \{1,2,\ldots,n-2\}$ is the map
\[\begin{array}{rcl}
	f(j) &=& \left\{ \begin{array}{l} j \textup{ if } j \leq i-1 \textup{ and } \\
						     j-2 \textup{ if } j \geq i+2\end{array} \right. \end{array}\]
then $a' = \{ (f(j), f(j')): (j,j') \in a' \textup{ and } (j,j') \neq (i,i+1)\}$.  Define the projection map $q': S'_{a,n} \rightarrow S'_{a',n-2}$ by $q'(x_1, x_2, \ldots, x_n) = (x_1, x_2, \ldots, \widehat{x_i}, \widehat{x_{i+1}}, \ldots, x_n)$ where $\widehat{x_j}$ omits the $j^{th}$ coordinate.

In \cite[Section 2.2]{Cautis-Kamnitzer} Cautis-Kamnitzer define a map $q: Z^i_n \rightarrow Y_{n-2}$ by \[V_{\bullet} \mapsto (V_1, \ldots, V_{i-1}, XV_{i+2}, XV_{i+3}, \ldots, XV_n)\] 
and prove that $Z^i_n$ is a ${\mathbb{P}}^1$-bundle over $Y_{n-2}$. 

\begin{lemma}
Let $a$ be a noncrossing matching on $n$ vertices containing the arc $(i,i+1)$ and let $a'$ be a noncrossing matching on $n-2$ vertices. Suppose $a$ is obtained from $a'$ by erasing the arc $(i,i+1)$.  Then there is a commutative diagram
\[\begin{array}{rcl}
C_{a,n} & \stackrel{\ell}{\rightarrow} & S'_{a,n} \\
\textup{\tiny{$q$}} \downarrow & & \downarrow {\textup{\tiny{$q'$}}} \\
Y_{n-2} &  \stackrel{\ell}{\rightarrow} & S'_{a',n-2}\end{array}\]
and the image $q(C_{a,n})$ is $C_{a',n-2}$.
\end{lemma}

\begin{proof}
We will show $q' \circ \ell = \ell \circ q$.  Then since $\ell$ is a diffeomorphism and $q' \circ \ell$ is surjective, the image $q(C_{a,n)}$ is diffeomorphic to $S'_{a',n}$.  By definition $q(C_{a,n}) = C_{a',n-2}$.

By construction the $i^{th}$ and $(i+1)^{th}$ coordinates of $S'_{a,n}$ are antipodes of each other so $C_{a,n} \subseteq Z^i_n$.  If $V_{\bullet} \in Z^i_n$ then $V_{i+1} = X^{-1}V_{i-1}$ and in particular $\ker X \subseteq V_{i+1}$.  Since the line $L_j$ is defined so that $L_j \perp V_{j-1}$ and $V_j = L_j \oplus V_{j-1}$ for each $j$, we conclude that $L_j \perp \ker X$ for $j \geq i+2$.  Thus $L_j$ is spanned by $\sum_{i \geq 2} a_i e_i + b_i f_i$ for some coefficients $a_i, b_i$.  So $C(L_j) = C(XL_j)$ for all $j \geq i+2$.

We confirm that $XL_j \perp XV_{j-1}$ for $j \geq i+2$.  The line $L_j$ is the span of $\sum_{i \geq 2} a_i e_i + b_i f_i$ so if $w \in V_{j-1}$ then since $e_i, f_i$ are an orthonormal basis the inner product satisfies 
\[\left \langle \sum_{i \geq 2} a_i e_i + b_i f_i, w \right \rangle = \left \langle \sum_{i \geq 2} a_i e_{i-1} + b_i f_{i-1}, Xw \right \rangle.\]
This shows that if $V_{\bullet} \in C_{a,n}$ then
\[\begin{array}{rcl}
\ell(q(V_{\bullet})) &=& (C(L_1), C(L_2), \ldots, C(L_{i-1}), C(XL_{i+2}), C(XL_{i+3}), \ldots, C(XL_n)) \\
 &=& (C(L_1), C(L_2), \ldots, C(L_{i-1}), C(L_{i+2}), C(L_{i+3}), \ldots, C(L_n)) \\
 &=& q'(\ell(V_{\bullet})).\end{array}\]
\end{proof}

The next lemma proves that the $C_{a,n}$ are the components of the Springer variety $X_n$.

\begin{lemma}
The union $\bigcup_{a\in B^{n/2}} C_{a,n}$ is isomorphic to the $(n/2, n/2)$ Springer variety in the vector space $\ker X^{n/2}$.  Each $C_{a,n}$ is an irreducible component of the Springer variety.
\end{lemma}

\begin{proof}
The proof inducts on $n$.  When $n=2$ there is a unique crossingless matching $a' = (1,2)$.  In this case $C_{a',2} = Z^1_2 = \{V_{\bullet} \in Y_2: V_2 = X^{-1}V_0\}$.  Equivalently 
\[C_{a',2} = \{ V_1 \subseteq V_2: V_2 = \ker X \textup{ and } XV_j \subseteq V_j \textup{ for all }j\},\]
which means $C_{a',2}$ is isomorphic to the $(1,1)$ Springer variety.

Assume that the claim holds for $n-2$.  Let $(i,j)$ be an arc in $a$ with minimal distance $j-i$; if $j-i \neq 1$ then $a$ is not a noncrossing matching.  Let $a'$ be the matching with $(i,i+1)$ removed, so $C_{a,n} \stackrel{q}{\rightarrow} C_{a',n-2}$.  If $q(V_{\bullet}) = V_{\bullet}'$ then for all $j$ there exists $k$ with $V_j \subseteq X^{-1} V_k'$.  By the inductive hypothesis each $V_k' \subseteq \ker X^{(n-2)/2}$ and so each $V_j \subseteq \ker X^{n/2}$.  Thus for all noncrossing matchings $a$, each $C_{a,n}$ is contained in the $(n/2, n/2)$ Springer variety inside $\ker X^{n/2}$.  Since $\ell$ is a diffeomorphism we conclude that all $C_{a,n}$ are compact irreducible subvarieties of the Springer variety of dimension $\dim X_n$, namely each $C_{a,n}$ is a component of the Springer variety.  Moreover if $a \neq b$ then $C_{a,n} \neq C_{b,n}$ because $\ell$ is a diffeomorphism.  The noncrossing matchings are in bijection with the components of the Springer variety so $\bigcup_{a \in B^{n/2}} C_{a,n}$ is isomorphic to the $(n/2, n/2)$ Springer variety in the vector space $\ker X^{n/2}$.
\end{proof}

The previous two lemmas allow us to conclude:

\begin{theorem}
The $(n/2, n/2)$ Springer variety $X_n$ is diffeomorphic to $\bigcup_{a \in B^{n/2}} S'_{a,n}$.
\end{theorem}

Finally we observe this is diffeomorphic to Khovanov's construction of the Springer variety.

\begin{theorem}
The $(n/2, n/2)$ Springer variety $X_n$ is diffeomorphic to $\bigcup_{a \in B^{n/2}} S_{a,n}$.
\end{theorem}

\begin{proof}
We prove that $\bigcup_{a \in B^{n/2}} S'_{a,n}$ is diffeomorphic to $\bigcup_{a \in B^{n/2}} S_{a,n}$ via the antipodal map $\gamma$ of Section \ref{antipodal embeddings}.  If $a$ is a noncrossing matching then each arc $(i,j) \in a$ has exactly one even and one odd endpoint by Lemma \ref{parity}, so $\gamma$ will be the identity on exactly one coordinate $i,j$ and the antipodal map on the other.  Hence if $(x_1, \ldots, x_n) \in \bigcup_{a \in B^{n/2}} S_{a,n}$ then $\gamma(x_1, \ldots, x_n) \in \bigcup_{a \in B^{n/2}} S'_{a,n}$.  The differentiable map $\gamma$ is its own inverse and so is a diffeomorphism.
\end{proof}

\end{document}